\documentclass[11pt]{article}
\usepackage{graphics}
\usepackage{latexsym}
\usepackage{amsfonts}
\usepackage{epstopdf}
\usepackage{caption}
\usepackage{subcaption}
\usepackage{graphicx}
\usepackage{amsmath}
\usepackage{amsthm}
\usepackage{amssymb}
\newtheorem{definition}{Definition}
\newtheorem{theorem}{Theorem}
\newtheorem{example}{Example}

\begin{document}
	
	\begin{titlepage}
		\Large
		\begin{center}
			{\bf \Huge
			QSR-Dissipativity and Passivity Analysis of Event-Triggered Networked Control Cyber-Physical Systems
			}
		\end{center}
		
		\begin{center}
				Technical Report of the ISIS Group\\
				at the University of Notre Dame\\
				ISIS-2016-002\\ 
				February 2016 
		\end{center}
		\vspace{1.5in}
		
		\begin{center}
			
			Arash Rahnama$^1$, Meng Xia$^1$, and Panos J. Antsaklis$^1$\\
			$^1$Department of Electrical Engineering\\
			University of Notre Dame\\
			Notre Dame, IN 46556\\
		\end{center}
		\vspace{1in}
		\begin{center}
			{\bf \Large Interdisciplinary Studies in Intelligent Systems}
		\end{center}
		\vspace{0.5in}
		\begin{center}
			{\textbf {Acknowledgements}    The support of the National Science Foundation under the CPS Grant No. CNS-1035655 and CNS-1446288 is gratefully acknowledged.
			}
		\end{center}
	\end{titlepage}

	\begin{abstract}
		Input feed-forward output feedback passive (IF-OFP) systems define a great number of dynamical systems. In this report, we show that dissipativity and passivity-based control combined with event-triggered networked control systems (NCS) provide a powerful platform for the design of cyber-physical systems (CPS). We propose \emph{QSR-dissipativity}, passivity and $L_2$-stability conditions for an event-triggered networked control system in three cases where: (i) an input-output event-triggering sampler condition is located on the plant's output side, (ii) an input-output event-triggering sampler condition is located on controller's output side, (iii)  input-output event-triggering sampler conditions are located on the outputs of both the plant and controller. We will show that this leads to a large decrease in communicational load amongst sub-units in networked control structures. We show that passivity and stability conditions depend on passivity levels for the plant and controller. Our results also illustrate the trade-off among passivity levels, stability, and system's dependence on the rate of communication  between the plant and controller.
	\end{abstract}
	\clearpage
	\section{Introduction}
	
	While classical control relies on the study of interconnected dynamical systems with dedicated communicational links, networked control systems rely on interconnected systems that communicate over channels. Hence, networked Control Systems' lie at the intersection of control and communication theory and are greatly useful for the design of Cyber-Physical Systems (CPS). As a form of spatially distributed systems consisting of many sub-units, NCS support the communication between sensors, actuators, and controllers through a shared communication network 	(figure \ref{fig:NCS}). This means that under NCS framework, sensors and actuators attached to the plant communicate with a remote controller over a multi-purpose shared network resulting in a flexible architecture with a reducing cost in installation and maintenance. In the recent years, and given the modern technological advances including the production of cheap, small, and efficient power processors with great sensing, information processing, and communication capabilities, NCS have been attracting a significant amount of interest in academic research and industry replacing the traditional control systems that rely on dedicated connections between sensor, controllers, and actuators. Due to its low cost of installation, flexibility in design, and ease of maintenance, NCS are identified as one the promising future direction for control \cite{murray2003future,baillieul2007control}, and have been finding application in a broad range of areas such as unmanned aerial vehicles \cite{seiler2005h}, remote surgery \cite{meng2004remote}, mobile sensor networks \cite{ogren2004cooperative}, and haptics collaboration over the Internet\cite{ogren2004cooperative,hikichi2002evaluation}.  A survey on recent developments in this field can be found in \cite{hespanha2007survey}.
	
	However, current NCS designs suffer from several challenges that are raised due to the innate limitations in communication networks such as sampling, quantization, packet dropouts, and delays that are resulted from continuous signals being encoded in a digital format, transmitted over the channels, and then being decoded at the receiving end. The delay in this process can be highly variable due to variant network access times and transmission delays. It is also possible that data may be lost while in transition through networks. Some results in the literature dealing with the problem of delays propose upper bounds for allowable delays in NCS called the maximum allowable transfer interval (MATI) \cite{walsh2002stability}, some of other possible solutions are given in \cite{branicky2000stability,lin2003robust}. Additionally, packet dropouts are usually modeled as stochastic incidents such as Bernoulli process \cite{sinopoli2004kalman} or finite-state Markov chains \cite{smith2003estimation}. Ways of dealing with quantization relay on dynamical and static quantization methods. A static quantizer is memoryless with fixed quantization levels, and dynamic quantizers are more complicated and use memory to adopt their quantization levels \cite{brockett2000quantized,fu2005sector}. Other issues related to NCS include security and safety of these interconnections. Additionally, most results in the literature focus on stability, and issues related to performance have been largely overlooked. Even most works on stability rely on worst case scenarios leading to rather conservative results.  Moreover, any communication network can only carry a finite amount of information per unit of time. This puts significant amount of constraint on the operation of NCS so any design that would decrease the plants' reliance on constant communication with the controllers is of great interest.
	
     In contrast to the classic periodical sensing and actuation control \cite{franklin1994feedback}, under an event-based framework, information between the plant and controller is only exchanged when it is necessary. This usually happens when a certain controlled value in the system deviates from its desired value for larger than a certain threshold. In other words, the information is exchanged only when something "significant" happens, and control is not executed unless it is required. Under this framework, one can obtain most control objectives by an open loop controller, while uncertainty is inevitable in real systems, a close-loop event-based framework can robustly deal with these uncertainties. Event-based control over networks has regained research interest since it creates a better balance between control performance and communication, and computational load  compared to the time-based counterpart \cite{aastrom1999comparison}. This shows that event-based control has a great potential for decreasing the bandwidth requirements for the network \cite{henningsson2006event}. This motivates the development of various event-based networked control schemes  \cite{donkers2010output,heemels2015event,tolic2015input,molin2010optimal}. A comprehensive survey on event-triggered control is given in \cite{lemmon2010event}.

 Overall, the shift of control technology to NCS motivates us to consider control and communication in a unified way. The theory of dissipativity, and QSR-dissipativity and passivity in particular can be used as a unifying force in the design of NCS. Passivity and dissipativity encompass the energy consumption characterizations of a dynamical system. Passivity is preserved under parallel and feedback interconnections \cite{bao2007process}. Passivity also implies stability under mild assumptions \cite{bao2007process,khalil2002nonlinear} making them a great alternative for designing compositional large-scale control systems. Additionally, the centralized control system approaches are not suitable for the design of CPS as they do not meet the basic requirements such as decentralization of control, integrated diagnostics, quick and easy maintenance and low cost. Motivated by the previous works of our colleagues in   \cite{yu2013event,zhu2014passivity,zhu2014passivityy}, and our recent work in regard to passivation of finite-gain nonlinear and linear systems \cite{xia2014passivity, xia2015guaranteeing}, and based on innate flexible qualities of passivity and event-triggered  networked control systems, we propose that combining both passivity and event-based control will lead to a suitable systematic asynchronous design framework for the control of large-scale cyber-physical systems.
 
 This report is organized as the following: Section \ref{sec:qsrpassivty} gives the preliminarily mathematical definitions on dissipativity and passivity. The problem statement being addressed is clearly stated in \ref{sec:prob}. Section \ref{sec:main} includes our main results with illustrative examples. Section \ref{sec:con} concludes the report and gives a brief overview of our future work.  
 
	\begin{figure}[!t]
		\centering
		\includegraphics[scale = 0.8]{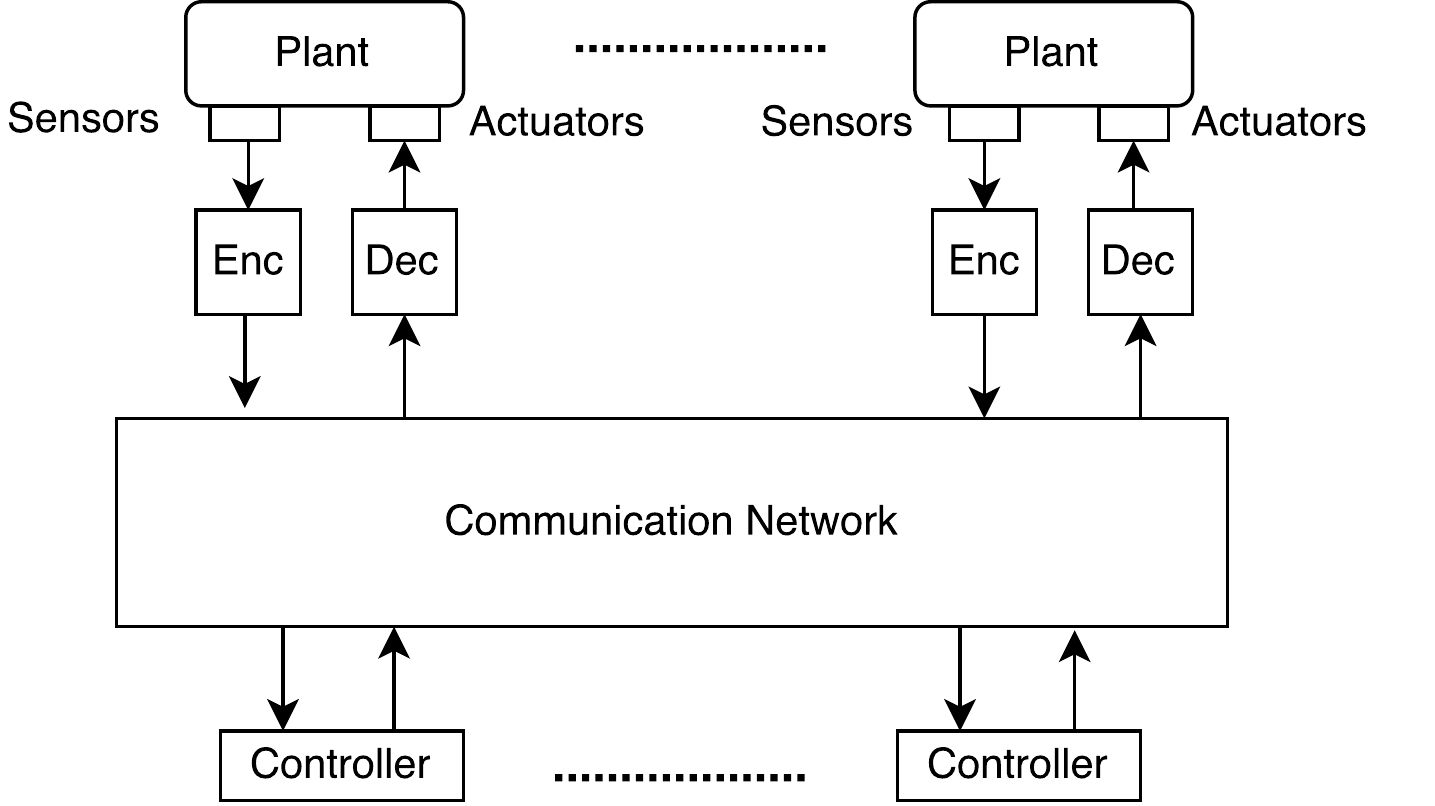}
		\caption{A Networked Control System Framework.}
		\label{fig:NCS}
	\end{figure}
	
	\clearpage
	\section{Mathematical Background}
	\label{sec:qsrpassivty}
	Consider the following linear or nonlinear dynamical system $G$, 
	\begin{align}
	\label{dynsys}
	G:
	\begin{cases}
	\dot{x}(t)= f(x(t),u(t) &  \\
	y(t)= h(x(t),u(t)), & 
	\end{cases}
	\end{align}
	where $x(t) \in X \subset R^n$, and $u(t) \in U \subset R^m$, and $y(t) \in Y \subset R^m$  are respectively the state, input and output of the system, and $X$, $U$ and $Y$ are respectively the state, input and output spaces. 
	Dissipativity and passivity are energy-based notions that characterize a dynamical system by its input/output behavior. A system is dissipative or passive if the increase in the system's stored energy is less than the entire energy supplied to it. The energy supplied to the system is defined by the storage function, and the energy stored in the system is defined by the storage function:
	\begin{definition}
		\label{diss}(\cite{willems1972dissipative}) System $G$ is dissipative with respect to the well-defined supply rate $\omega(u,y)$, if there exists a nonnegative storage function $V(x): X \to R^+$ such that for all $t_0$, $t_1$ where $t_1\geq t_0$, and all solutions $x(t)=x \in X$, $u(t)=u\in U$, $y(t)=y \in Y$:
		\begin{align}
		\label{eq:dissipative}
		V(x(t_1))-V(x(t_0)) \leq \int_{t_0}^{t_1}\omega(u(t),y(t)) dt 
		\end{align}
		is satisfied. If the storage function is differentiable, then \eqref{eq:dissipative} can be written as:
		\begin{align}
		\label{eq:dissipative2}
		\dot{V}(x(t)) \leq \omega(u(t),y(t)),~ \forall t 
		\end{align} 
	\end{definition}
	Accordingly, we call a system \emph{QSR-disspative} if it is dissipative with respect to the well-defined supply rate:
	\begin{align}
	\label{eq:QSR}
	\omega(u,y) = y^T Q y + 2 y^T S u + u^T R u, 
	\end{align}
	where $Q$, $R$, and $S$ are constant matrices of appropriate dimensions, and $Q$ and $R$ are symmetric \cite{hill1976stability}. 
	\begin{definition}(\cite{khalil2002nonlinear}) System $G$ is called finite gain $L_2$-stable, if there exists a positive-semi definite function $V(x): X \to R^+$, and a scalar constant $\gamma>0$ such that for all $u(t)=u\in U$, $y(t)=y \in Y$ and $t_1>0$
		\begin{align}
		\label{eq:l2}
		V(x(t_1))-V(x(0)) \leq \int_{0}^{t_1}(\gamma ^2 u^T(t)u(t)-y^T(t)y(t)) dt 
		\end{align}
	\end{definition}
	The relation between \emph{QSR-disspativity} and $L_2$-stability is well-established:
	\begin{theorem}(\cite{hill1976stability})
		If system $G$ is \emph{QSR-disspative} with $Q<0$, then it is $L_2$-stable.
	\end{theorem}
	
	\begin{definition}(\cite{bao2007process}) As a special case of dissipativity, system $G$ is called passive, if there exists a nonnegative storage function $V(x): X \to R^+$ such that:
		\label{pass}
		\begin{align}
		\label{eq:passivity}
		V(x(t_1))-V(x(t_0)) \leq \int_{t_0}^{t_1} u^T(t)y(t) dt 
		\end{align}
		is satisfied for all $t_0$, $t_1$ where $t_1\geq t_0$, and all solutions $x(t)=x \in X$, $u(t)=u\in U$, $y(t)=y \in Y$. If the storage function is differentiable, then \eqref{eq:passivity} can be written as:
		\begin{align}
		\label{eq:passivity2}
		\dot{V}(x(t)) \leq u^T(t)y(t), \forall t 
		\end{align}
	\end{definition}
	Under certain conditions, passivity coincides with input/output stability, and for zero-state detectable dynamical systems, guarantees the stability of the origin \cite{khalil2002nonlinear}.
	\begin{definition}(\cite{khalil2002nonlinear}) System $G$ is considered to be Input Feed-forward Output Feedback Passive (IF-OFP), if it is dissipative with respect to the well-defined supply rate:
		\label{IFOF}
		\begin{align}
		\label{eq:IFOFP}
		\omega(u,y)=u^Ty-\rho y^Ty-\nu u^Tu, \forall t\geq0,
		\end{align}
		for some $\rho, \nu \in R$.
	\end{definition}
	
	IF-OFP property presents a more general form for the concept of passivity. Based on definition \ref{IFOF}, we can denote an IF-OFP system with IF-OFP($\nu$,$\rho$). $\nu$ is called the input passivity index and $\rho$ is called the output passivity index. Passivity indices are a means to measure the shortage and excess of passivity in dynamical systems \cite{khalil2002nonlinear}, and are useful in passivity-based analysis and control of systems \cite{yu2013event, xia2014passivity}. A positive value for either one of two passivity indices points to an excess in passivity; and a negative value for either of two passivity indices points to a shortage in passivity.  An excess of passivity in one system can compensate for the shortage of passivity in another system leading to a passive feedback or feed-forward interconnection \cite{hirche2012human}. Moreover, passivity indices can be useful for analyzing the performance of passive systems. Further, if only $\nu>0$, then the system is said to be \emph{input strictly passive} (ISP); if only $\rho>0$, then the system is said to be \emph{output strictly passive} (OSP). Similarly, if $\nu>0$ and $\rho>0$, then the system is said to be \emph{very strictly passive} (VSP). 
	\clearpage
\section{Problem Statement}
\label{sec:prob}
As shown in figures \ref{fig:ETNCSPS}, \ref{fig:ETNCSCS}, and \ref{fig:ETNCSPSCS}, we are exploring the interconnection of two IFOF passive systems. The main plant has passivity levels $\rho_p$, $\nu_p$, and the controller has passivity levels $\rho_c$, $\nu_c$. The indices can take positive or negative values indicating the extent that each sub-system is passive or non-passive. Most linear and nonlinear systems can be represented by this definition given that their passivity indices are known. The triggering mechanisms in the setups are representing the situations, in which new information is sent every time a violation of the triggering condition occurs. 

In figure \ref{fig:ETNCSPS}, an event-detector is located on the output of the plant to monitor the behavior of plant's output. An updated measure of $y_p$ is sent to the communication network when the error between the last information sent ($y_p(t_{k})$) and the current one  $e_p(t)=y_p(t)-y_p(t_k)$ (for $t \in [t_k, t_{k+1})$) exceeds a predetermined threshold established by the designer. A similar setup is also presented in figure \ref{fig:ETNCSCS} where the event-detector is located on the output of the controller, and an updated measure of $y_c$ is sent to the communication network when the error between the last information sent from the controller ($y_c(t_{k})$) and the current one  $e_c(t)=y_c(t)-y_c(t_k)$ (for $t \in [t_k, t_{k+1})$) exceeds a predetermined threshold value. In figure \ref{fig:ETNCSPSCS}, the event-detectors are located on both sides of the interconnection, this results in a great decrease in the amount of information required to be exchanged between sub-systems in order for them to maintain the desirable performance. In all cases the inputs to the controller or plant are held constant based on the last value received. The updating process is governed by the event-detectors. This structure is commonly used in literature to analyze the behavior of networked interconnections as it can capture different NCS configurations \cite{hespanha2007survey}. $w_1(t)$ denotes a reference input or an external disturbance on the plant side, and $w_2(t)$ denotes an external disturbance on the controller side. Additionally, we consider a simple triggering condition for each side:
\begin{align}
||e_p(t)||_2^2 > \delta_p||y_p(t)||_2^2 ~~~~ 0 < \delta_p \leq 1\\
||e_c(t)||_2^2 > \delta_c||y_c(t)||_2^2 ~~~~ 0 < \delta_c \leq 1
\end{align}

This will facilitate the design process and also make it easier for the designer to understand and analyze the trade-off between performance, finite-gain $L_2$-stability, \emph{QSR-disspativity}, passivity, and channel utilization, and to make design decisions accordingly. Another advantage of the above conditions is that we do not need the exact dynamical models for each sub-system before making design decisions, and that an access only to each sub-system's output is sufficient. 

In this paper, we will show \emph{QSR-disspativity}/passivity and finite-gain $L_2$-stability conditions for each interconnection under the triggering conditions mentioned above. We will seek to answer the following question: Given $G_p$, and $G_c$ how can we design an appropriate triggering interval so that we can efficiently utilize the band-limited communication channel meanwhile reaching a stable performance? We will show that the answer to this question depends on the passivity levels for each sub-unit, and the flexibility of our triggering conditions: passivity indices determine the communication rate by directly affecting the triggering conditions, and passivity has a direct relationship with performance, in the sense that systems that are more passive utilize the communication network less frequently meanwhile maintaining a stable performance.

\begin{figure}[!t]
	\centering
	\includegraphics[scale = 0.5]{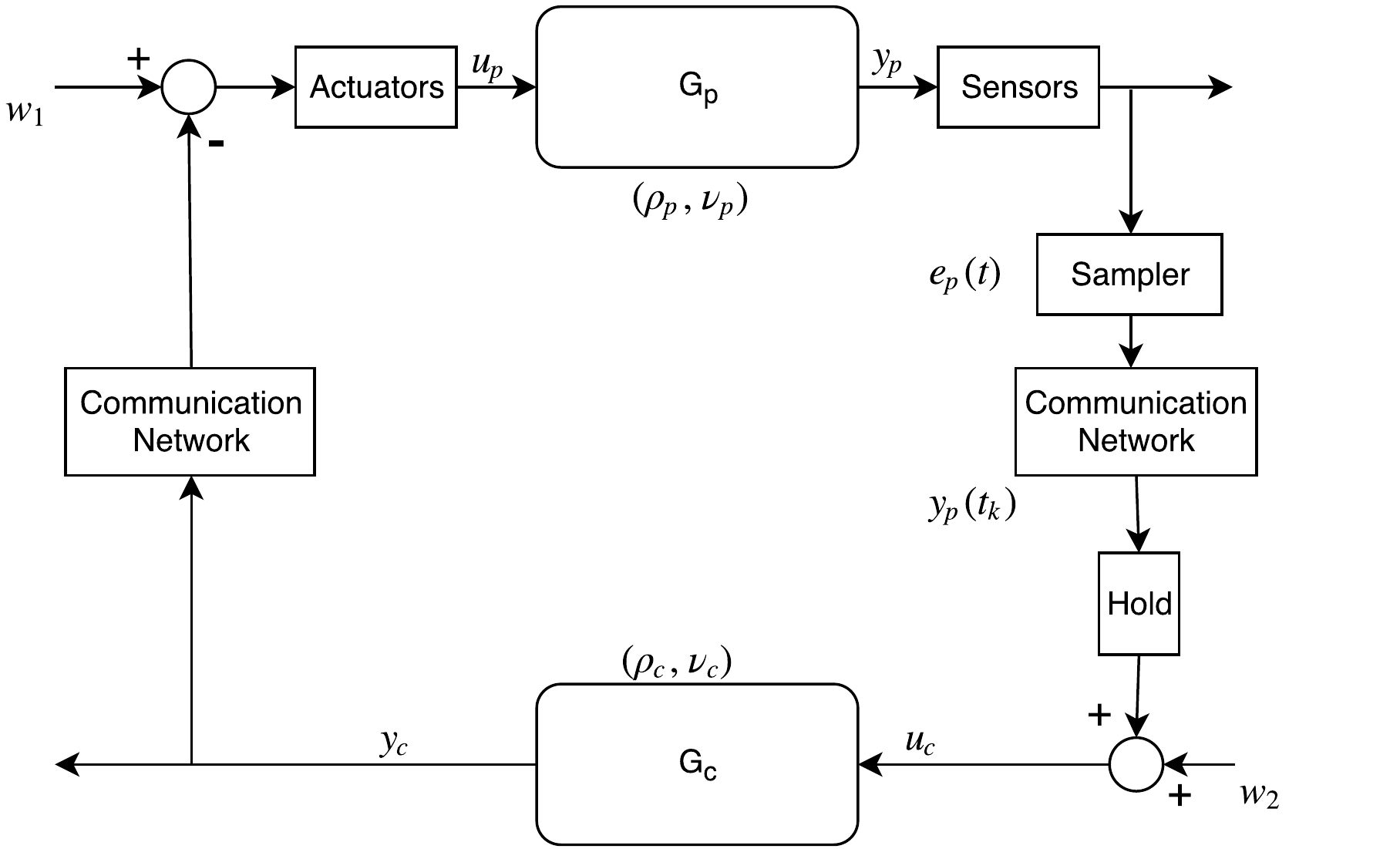}
	\caption{A Networked Control System Interconnection of two IF-OFP systems - the event-triggering condition on plant's output side.}
	\label{fig:ETNCSPS}
\end{figure}
\begin{figure}[!t]
	\centering
	\includegraphics[scale = 0.5]{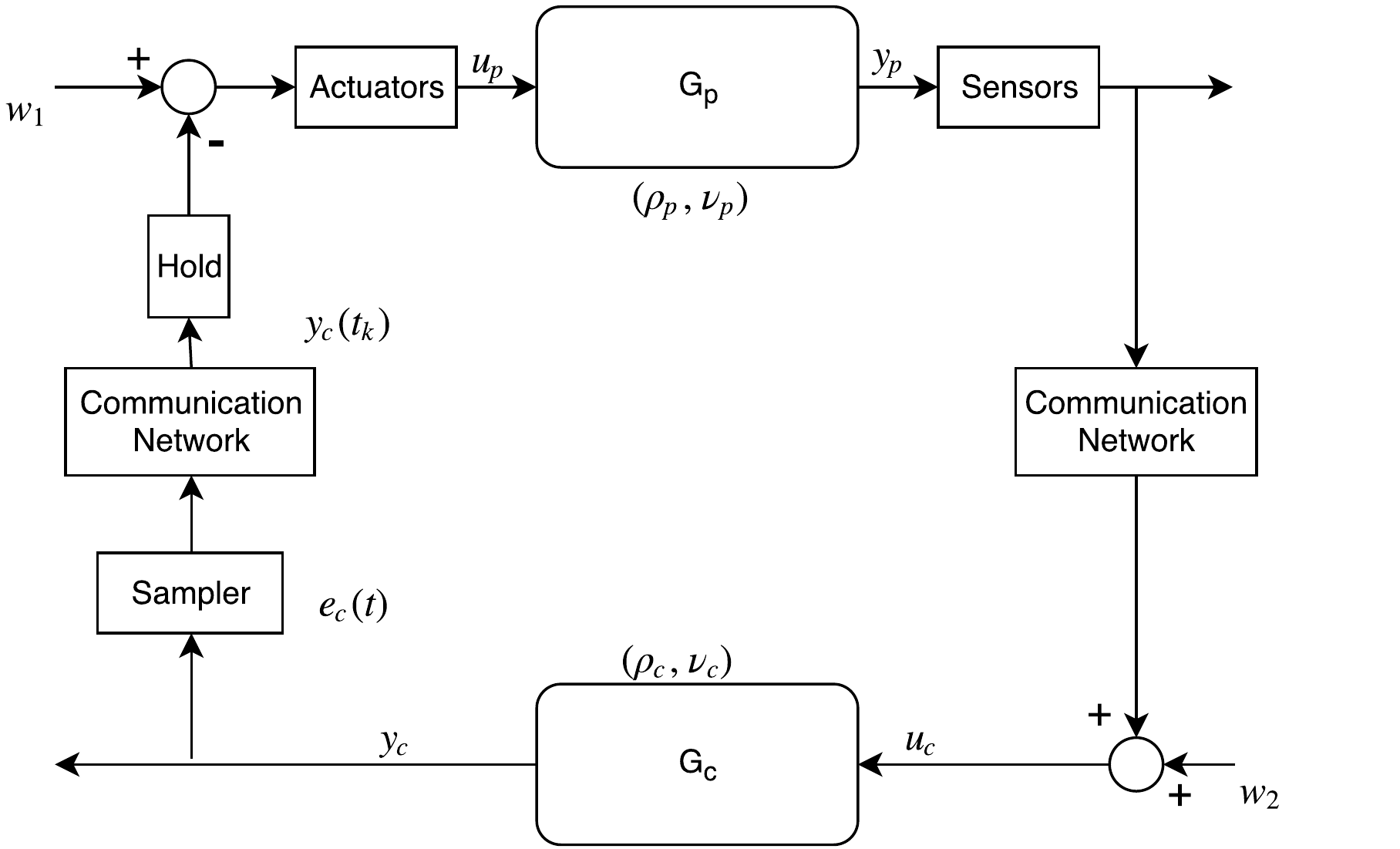}
	\caption{A Networked Control System Interconnection of two IF-OFP systems - the event-triggering condition on controller's output side.}
	\label{fig:ETNCSCS}
\end{figure}
\begin{figure}[!t]
	\centering
	\includegraphics[scale = 0.5]{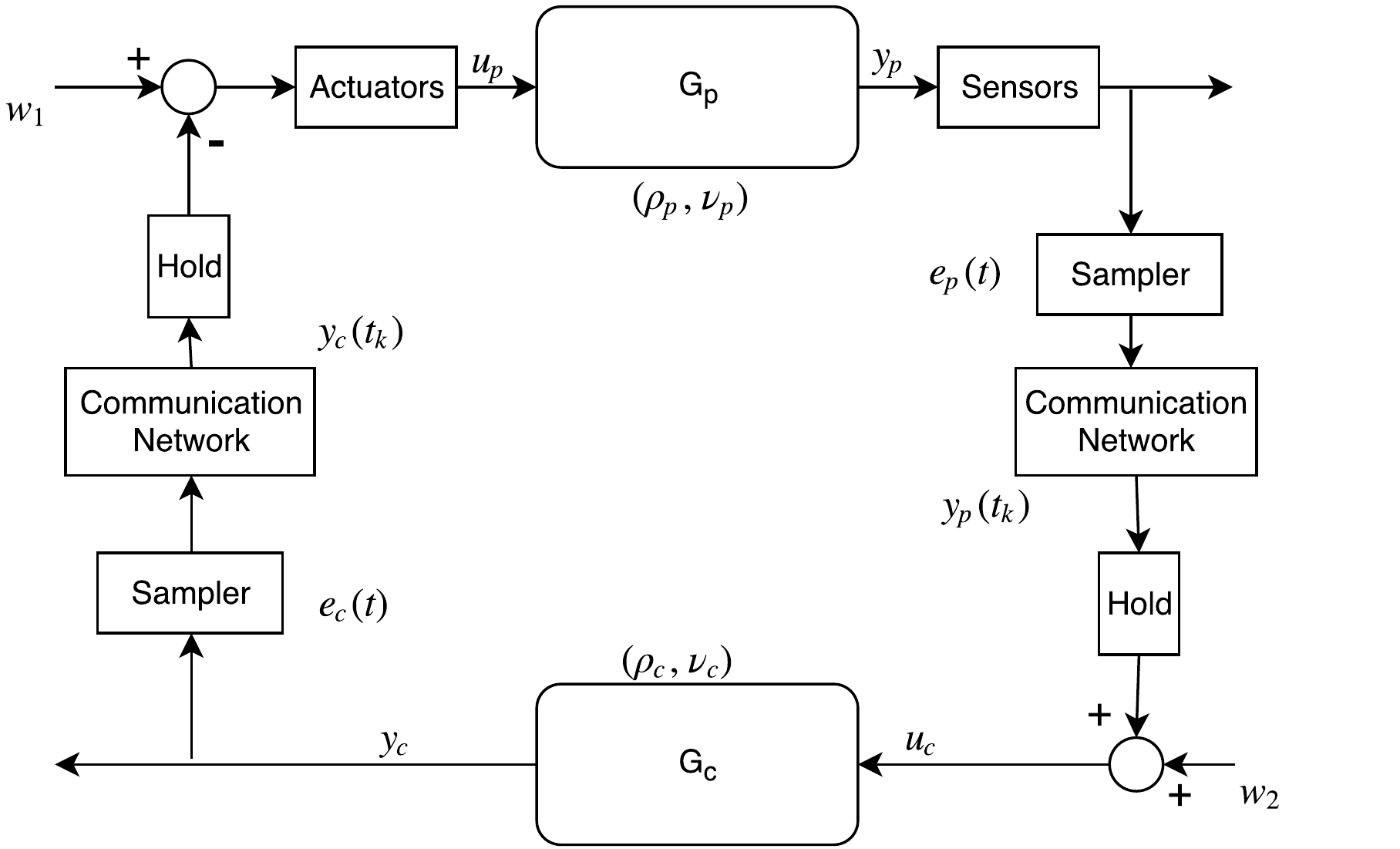}
	\caption{A Networked Control System Interconnection of two IF-OFP systems - the event-triggering conditions on plant's and controller's outputs.}
	\label{fig:ETNCSPSCS}
\end{figure}
\clearpage
\section{Main Results}
\label{sec:main}
\subsection{\emph{QSR-dissipativity} and Passivity Analysis - Event-triggering condition on the Plant's side - Figure \ref{fig:ETNCSPS})}
\subsubsection{\emph{QSR-dissipativity} for NCS (Event-triggering condition on Plant's Output - Figure \ref{fig:ETNCSPS}) }
\begin{theorem}
	Consider the feedback interconnection of two systems $G_p$ and $G_c$ in figure \ref{fig:ETNCSPS} with respective passivity indices of $\nu_p$, $\rho_p$, $\nu_c$ and $\rho_c$. If the event instance $t_k$ is explicitly determined by the triggering condition $||e_p(t)||_2^2>\delta||y_p(t)||_2^2$ where $\delta  \in(0,1]$, then the event-triggered networked control system is \emph{QSR-dissipative} with respect to the inputs $w(t)=\begin{bmatrix}
	w_1(t) \\ w_2(t)   
	\end{bmatrix}$, and the outputs
	$y(t)=\begin{bmatrix}
	y_p(t) \\ y_c(t)   
	\end{bmatrix}$, and satisfies the relation:
	\begin{align*}
		&\dot{V}(t) \leq w^T(t)Rw(t)+ 2w^T(t)Sy(t)+y^T(t)Qy(t)
	\end{align*}
	where,\\
	
	$S=\begin{bmatrix}
	\frac{1}{2}I & \nu_pI\\
	-\nu_cI & \frac{1}{2}I      
	\end{bmatrix}$, $R=\begin{bmatrix}
	-\nu_pI & 0I\\
	0I & -(\nu_c-|\nu_c|)I      
	\end{bmatrix}$, and $Q=\begin{bmatrix}
	-\beta(\nu_c)I & 0I\\
	0I & -(\rho_c+\nu_p-\frac{1}{4})I\end{bmatrix}$
	and:\\
	
	\[\beta(\nu_c) =
	\begin{cases}
	\rho_p-\nu_c-\delta(1+\nu_c)      & \quad \text{if } \nu_c > 0\\
	\rho_p-\delta       & \quad \text{if } \nu_c = 0\\
	\rho_p+2\nu_c-\delta(1-3\nu_c)  & \quad \text{if } \nu_c < 0\\
	\end{cases}
	\]
	
	Additionally, if $Q<0$ meaning $\beta(\nu_c)>0$ and $\rho_c+\nu_p-\frac{1}{4}>0$ then the interconnection is $L_2$-stable.
	
\end{theorem}

\begin{proof} 
	\label{thm:NCSYpEp}
	Given the systems $G_p$ and $G_c$ with passivity indices $\nu_p$, $\rho_p$, $\nu_c$ and $\rho_c$, there exists $V_p(t)$ and $V_c(t)$ such that:
	
	\begin{align*}
		& \dot{V}_p(t) \leq u^T_p(t)y_p(t)-\nu_pu^T_p(t)u_p(t)-\rho_p y^T_p(t)y_p(t) 
		\\&\dot{V}_c(t) \leq u^T_c(t)y_c(t)-\nu_c u^T_c(t)u_c(t)-\rho_c y^T_c(t)y_c(t).\\
	\end{align*}
	
	Additionally, according to the setup portrayed in figure \ref{fig:ETNCSPS}, the following relationships stand for $t \in [t_k,t_{k+1})$:
	
	\begin{align*}
		& u_p(t)=w_1(t)-y_c(t)
		\\&e_p(t)=y_p(t)-y_p(t_k)
		\\&u_c(t)=w_2(t)+y_p(t_k)=w_2(t)+y_p(t)-e_p(t),
	\end{align*}
	
	we design the triggering condition based on the following rule ($||e_p(t)||_2^2>\delta||y_p(t)||_2^2 $):
	
	\begin{align*}
		& \langle e_p,e_p\rangle>\delta \langle y_p,y_p\rangle, ~ 0 < \delta \leq 1,
	\end{align*}

	We consider the following storage function for the interconnection:
	
	\begin{align*}
		& V(t) = V_p(t) + V_c(t),
	\end{align*} 
	
	as a results, we have:
	\begin{align*}
		& \dot{V}(t)=\dot{V}_p(t)+\dot{V}_c(t)\\&~~~~~~ \leq u^T_p(t)y_p(t)-\nu_pu^T_p(t)u_p(t)-\rho_p y^T_p(t)y_p(t)+u^T_c(t)y_c(t)-\nu_c u^T_c(t)u_c(t)-\rho_c y^T_c(t)y_c(t).
	\end{align*} 
	
	We know that $u_p(t)=w_1(t)-y_c(t)$, $u_c(t)=w_2(t)+y_p(t)-e_p(t)$, consequently for any $t \in [t_k,t_{k+1})$ we have:
	
	\begin{align*}
		&\dot{V}(t) \leq (w_1(t)-y_c(t))^T y_p(t)-\nu_p(w_1(t)-y_c(t))^T(w_1(t)-y_c(t))-\rho_p y^T_p(t)y_p(t)\\&+(w_2(t)+y_p(t)-e_p(t))^Ty_c(t)-\nu_c (w_2(t)+y_p(t)-e_p(t))^T(w_2(t)+y_p(t)-e_p(t))-\rho_c y^T_c(t)y_c(t)
		\\&= w_1^T(t)y_p(t)+w_2^T(t)y_c(t)-\nu_pw_1^T(t)w_1(t)-\nu_cw_2^T(t)w_2(t)-(\rho_p+\nu_c)y^T_p(t)y_p(t)-(\rho_c+\nu_p)y^T_c(t)y_c(t)\\&+2\nu_pw^T_1(t)y_c(t)-2\nu_cw^T_2(t)y_p(t)+2\nu_cy^T_p(t)e_p(t)+2\nu_cw^T_2(t)e_p(t)-e_p^T(t)y_c(t)-\nu_ce_p^T(t)e_p(t)
	\end{align*} 
	
	Given that $2\nu_cw^T_2(t)e_p(t) \leq |\nu_c|w_2^T(t)w_2(t)+|\nu_c|e_p^T(t)e_p(t)$ we have:
	
	$\dot{V}(t) \leq 2\begin{bmatrix}
	w^T_1(t) & w^T_2(t)      
	\end{bmatrix}\begin{bmatrix}
	\frac{1}{2} & \nu_p\\
	-\nu_c & \frac{1}{2}      
	\end{bmatrix}\begin{bmatrix}
	y_p(t) \\ y_c(t)      
	\end{bmatrix}+\begin{bmatrix}
	w^T_1(t) & w^T_2(t)      
	\end{bmatrix}\begin{bmatrix}
	-\nu_p & 0\\
	0 & |\nu_c|-\nu_c      
	\end{bmatrix}\begin{bmatrix}
	w_1(t) \\ w_2(t)      
	\end{bmatrix}\\
	+\begin{bmatrix}
	y^T_p(t) & y^T_c(t)      
	\end{bmatrix}\begin{bmatrix}
	-(\rho_p+\nu_c) & 0\\
	0 & -(\rho_c+\nu_p)     
	\end{bmatrix}\begin{bmatrix}
	y_p(t) \\ y_c(t)      
	\end{bmatrix}+2\nu_cy^T_p(t)e_p(t)+|\nu_c|e^T_p(t)e_p(t)-e_p^T(t)y_c(t)\\-\nu_ce_p^T(t)e_p(t)$\\
	
	Additionally we have: $-e_p^T(t)y_c(t)=-( e_p(t) + \frac{1}{2} y_c(t))^2+ e_p^T(t)e_p(t) + 
	\frac{1}{4} y_c^T(t)y_c(t)$ , and $e_p^T(t)e_p(t) \leq \delta y_p^T(t)y_p(t)$ so we have:

	$\dot{V}(t) \leq 2\begin{bmatrix}
	w^T_1(t) & w^T_2(t)      
	\end{bmatrix}\begin{bmatrix}
	\frac{1}{2} & \nu_p\\
	-\nu_c & \frac{1}{2}      
	\end{bmatrix}\begin{bmatrix}
	y_p(t) \\ y_c(t)      
	\end{bmatrix}+\begin{bmatrix}
	w^T_1(t) & w^T_2(t)      
	\end{bmatrix}\begin{bmatrix}
	-\nu_p & 0\\
	0 & |\nu_c|-\nu_c      
	\end{bmatrix}\begin{bmatrix}
	w_1(t) \\ w_2(t)      
	\end{bmatrix}\\
	+\begin{bmatrix}
	y^T_p(t) & y^T_c(t)      
	\end{bmatrix}\begin{bmatrix}
	-(\rho_p+\nu_c-\delta) & 0\\
	0 & -(\rho_c+\nu_p-\frac{1}{4})     
	\end{bmatrix}\begin{bmatrix}
	y_p(t) \\ y_c(t)      
	\end{bmatrix}+2\nu_cy^T_p(t)e_p(t)+|\nu_c|e^T_p(t)e_p(t)-\nu_ce_p^T(t)e_p(t)~~~(11)$\\

	if $\nu_c > 0$ then we have:\\ 
	
	$2\nu_cy^T_p(t)e_p(t)-\nu_ce^T_p(t)e_p(t)=-\begin{bmatrix}
	e_p(t) & y_p(t)      
	\end{bmatrix} M \begin{bmatrix}
	e_p(t) \\ y_p(t)      
	\end{bmatrix} + 2 \nu_c y_p^T(t)y_p(t)
	$   \\
	where $M=\begin{bmatrix}
	\nu_c & -\nu_c\\
	-\nu_c & 2\nu_c    
	\end{bmatrix}\geq0$, and $\nu_c e_p^T(t)e_p(t) \leq \delta \nu_c y_p^T(t)y_p(t)$ so we can simplify (11) further to have:
	
	$\dot{V}(t) \leq 2\begin{bmatrix}
	w^T_1(t) & w^T_2(t)      
	\end{bmatrix}\begin{bmatrix}
	\frac{1}{2} & \nu_p\\
	-\nu_c & \frac{1}{2}      
	\end{bmatrix}\begin{bmatrix}
	y_p(t) \\ y_c(t)      
	\end{bmatrix}+\begin{bmatrix}
	w^T_1(t) & w^T_2(t)      
	\end{bmatrix}\begin{bmatrix}
	-\nu_p & 0\\
	0 & |\nu_c|-\nu_c      
	\end{bmatrix}\begin{bmatrix}
	w_1(t) \\ w_2(t)      
	\end{bmatrix}\\
	+\begin{bmatrix}
	y^T_p(t) & y^T_c(t)      
	\end{bmatrix}\begin{bmatrix}
	-(\rho_p-\nu_c-\delta(1+\nu_c)) & 0\\
	0 & -(\rho_c+\nu_p-\frac{1}{4})     
	\end{bmatrix}\begin{bmatrix}
	y_p(t) \\ y_c(t)   
	\end{bmatrix}$
	
	if $\nu_c = 0$ then we have: \\
	
	$\dot{V}(t) \leq 2\begin{bmatrix}
	w^T_1(t) & w^T_2(t)      
	\end{bmatrix}\begin{bmatrix}
	\frac{1}{2} & \nu_p\\
	-\nu_c & \frac{1}{2}      
	\end{bmatrix}\begin{bmatrix}
	y_p(t) \\ y_c(t)      
	\end{bmatrix}+\begin{bmatrix}
	w^T_1(t) & w^T_2(t)      
	\end{bmatrix}\begin{bmatrix}
	-\nu_p & 0\\
	0 & |\nu_c|-\nu_c      
	\end{bmatrix}\begin{bmatrix}
	w_1(t) \\ w_2(t)      
	\end{bmatrix}\\
	+\begin{bmatrix}
	y^T_p(t) & y^T_c(t)      
	\end{bmatrix}\begin{bmatrix}
	-(\rho_p-\delta) & 0\\
	0 & -(\rho_c+\nu_p-\frac{1}{4})     
	\end{bmatrix}\begin{bmatrix}
	y_p(t) \\ y_c(t)   
	\end{bmatrix}$
	
	if $\nu_c < 0$, given that $2\nu_cy^T_p(t)e_p(t) \leq |\nu_c|y_p^T(t)y_p(t)+|\nu_c|e_p^T(t)e_p(t)$, $|\nu_c|e_p^T(t)e_p(t) \leq \delta|\nu_c| y_p^T(t)y_p(t)$ and $-\nu_c e_p^T(t)e_p(t) \leq \delta |\nu_c| y_p^T(t)y_p(t)$ then we have: \\
	
	$\dot{V}(t) \leq 2\begin{bmatrix}
	w^T_1(t) & w^T_2(t)      
	\end{bmatrix}\begin{bmatrix}
	\frac{1}{2} & \nu_p\\
	-\nu_c & \frac{1}{2}      
	\end{bmatrix}\begin{bmatrix}
	y_p(t) \\ y_c(t)      
	\end{bmatrix}+\begin{bmatrix}
	w^T_1(t) & w^T_2(t)      
	\end{bmatrix}\begin{bmatrix}
	-\nu_p & 0\\
	0 & |\nu_c|-\nu_c      
	\end{bmatrix}\begin{bmatrix}
	w_1(t) \\ w_2(t)      
	\end{bmatrix}\\
	+\begin{bmatrix}
	y^T_p(t) & y^T_c(t)      
	\end{bmatrix}\begin{bmatrix}
	-(\rho_p+2\nu_c-\delta(1-3\nu_c)) & 0\\
	0 & -(\rho_c+\nu_p-\frac{1}{4})     
	\end{bmatrix}\begin{bmatrix}
	y_p(t) \\ y_c(t)   
	\end{bmatrix}$
\end{proof}

Hence, we have shown that:

\begin{align*}
	&\dot{V}(t) \leq w^T(t)Rw(t)+2w^T(t)Sy(t)+y^T(t)Qy(t)
\end{align*}

where, $w(t)=\begin{bmatrix}
w_1(t) \\ w_2(t)   
\end{bmatrix}$,
$y(t)=\begin{bmatrix}
y_p(t) \\ y_c(t)   
\end{bmatrix}$ and:\\

$S=\begin{bmatrix}
\frac{1}{2}I & \nu_pI\\
-\nu_cI & \frac{1}{2}I      
\end{bmatrix}$, $R=\begin{bmatrix}
-\nu_pI & 0I\\
0I & -(\nu_c-|\nu_c|)I      
\end{bmatrix}$, and $Q=\begin{bmatrix}
-\beta(\nu_c)I & 0I\\
0I & -(\rho_c+\nu_p-\frac{1}{4})I\end{bmatrix}$

and:\\

\[ \beta(\nu_c) =
\begin{cases}
\rho_p-\nu_c-\delta(1+\nu_c)      & \quad \text{if } \nu_c > 0\\
\rho_p-\delta       & \quad \text{if } \nu_c = 0\\
\rho_p+2\nu_c-\delta(1-3\nu_c)  & \quad \text{if } \nu_c < 0\\
\end{cases}
\]
\subsubsection{Simulation Examples for \emph{QSR-dissipativity} for NCS (Event-triggering condition on the Plant's Output - Figure \ref{fig:ETNCSPS})}
\begin{example}
	The plant in figure \ref{fig:NCSYpEpEX1} is defined by the following model:	
	
	\begin{figure}[!t]
		\centering
		\includegraphics[scale = .65]{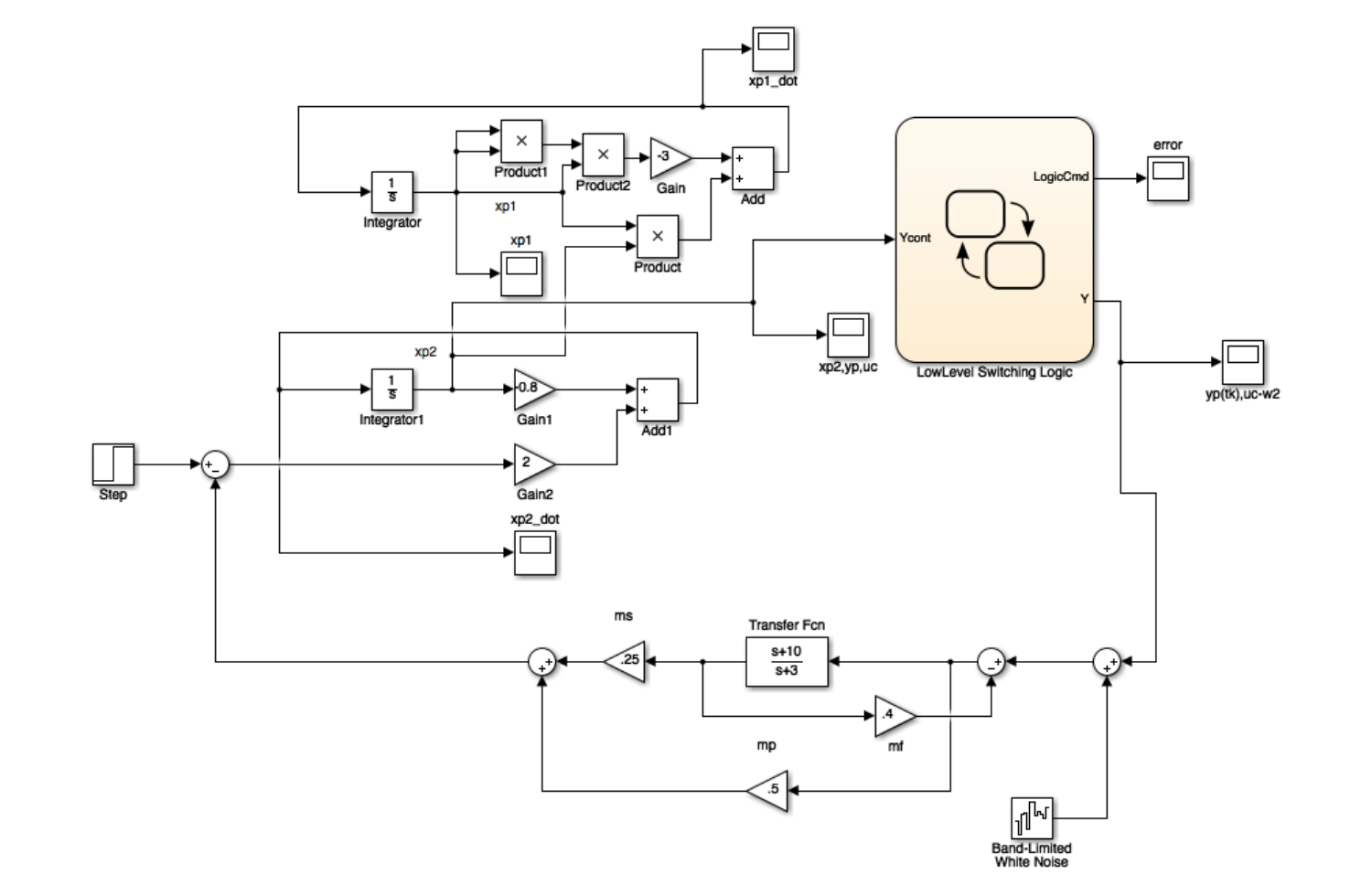}
		\caption{Simulink model for example 1.}
		\label{fig:NCSYpEpEX1}
	\end{figure}
	
	\begin{align*} 
		&\dot{x}_{p1}(t)=-3x^3_{p1}(t)+x_{p1}(t)x_{p2}(t)\\
		&\dot{x}_{p2}(t)=-0.8x_{p2}(t)+2u_p(t)\\
		&y_p(t)=x_{p2}(t),
	\end{align*}
	
	where $\rho_p=0.4$ and $\nu_p=0$. And the model for the controller accompanied with our passivation method (M matrix) is the following:
	
	\begin{align*} 
		&\dot{x}_{c}(t)=-3x_{c}(t)+u_{c}(t)\\
		&y_c(t)=7x_{c}(t)+u_c(t),
	\end{align*}
	
	after adding the M-matrix, $\rho_c=1.8$ and $\nu_c=0$. The interconnection satisfies $\rho_c+\nu_p-\frac{1}{4}>0$, and by picking $\delta=0.3$, $\rho_p-\delta >0$ is satisfied, $w1$ is a step function, and $w2$ is white noise with power $0.02$, the simulation results are given in: figures  \ref{fig:NCSYpEpEX1XP1}, \ref{fig:NCSYpEpEX1XP2}, \ref{fig:NCSYpEpEX1Y}.
	\begin{figure}[!t]
		\centering
		\includegraphics[scale = .5]{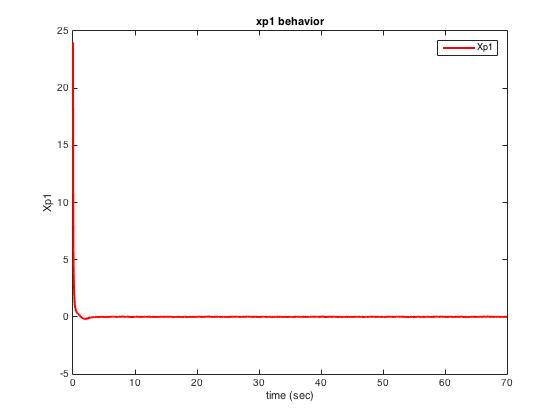}
		\caption{Simulation results for example 1.}
		\label{fig:NCSYpEpEX1XP1}
	\end{figure}
	\begin{figure}[!t]
		\centering
		\includegraphics[scale = .5]{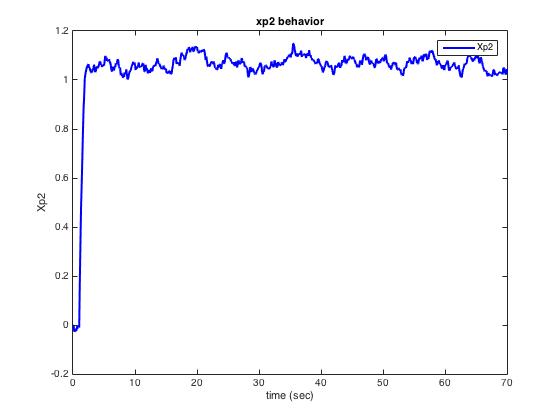}
		\caption{Simulation results for example 1.}
		\label{fig:NCSYpEpEX1XP2}
	\end{figure}
	\begin{figure}[!t]
		\centering
		\includegraphics[scale = .5]{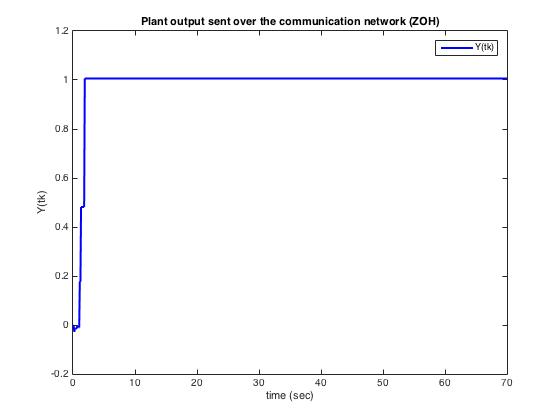}
		\caption{Simulation results for example 1.}
		\label{fig:NCSYpEpEX1Y}
	\end{figure}
\end{example}
\begin{example}
	The plant is defined by the following model:	
	\begin{align*} 
		&\dot{x}_{p1}(t)=x_{p2}(t)\\
		&\dot{x}_{p2}(t)=-0.5x^3_{p1}(t)-x_{p2}(t)+u_p(t)\\
		&y_p(t)=x_{p2}(t),
	\end{align*}
	
	where $\rho_p=1$ and $\nu_p=0$. And the model for the controller is the following:
	
	\begin{align*} 
		&\dot{x}_{c1}(t)=-2x_{c1}(t)-x_{c2}(t)+u_{c}(t)\\
		&\dot{x}_{c2}(t)=-3x_{c1}(t)-5x_{c2}(t)+2u_{c}(t)\\
		&y_c(t)=x_{c1}(t)+x_{c2}(t)+u_c(t),
	\end{align*}
	
	where $\rho_c=0.5$ and $\nu_c=0.3$. The interconnection satisfies $\rho_c+\nu_p-\frac{1}{4}>0$, and by picking $\delta=0.5$, $\rho_p-\nu_c-\delta(1+\nu_c)>0$ is satisfied, $w_1=sin(\frac{5\pi}{2}t)$ is a sinusoidal signal, and $w2$ is white noise with power $0.0001$, the simulation results are given in: figures  \ref{fig:NCSYpEpEX2Ytk}, \ref{fig:NCSYpEpEX2Y}.

	\begin{figure}[!t]
		\centering
		\includegraphics[scale = .5]{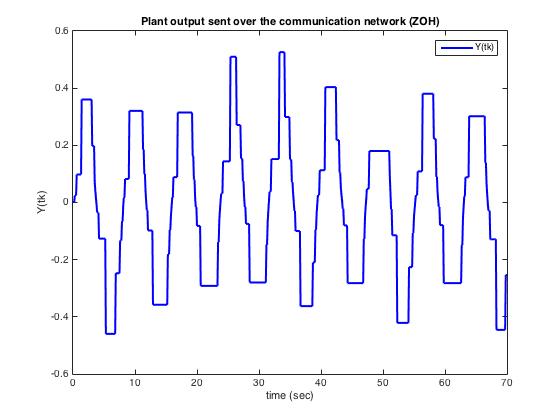}
		\caption{Simulation results for example 2.}
		\label{fig:NCSYpEpEX2Ytk}
	\end{figure}
	\begin{figure}[!t]
		\centering
		\includegraphics[scale = .5]{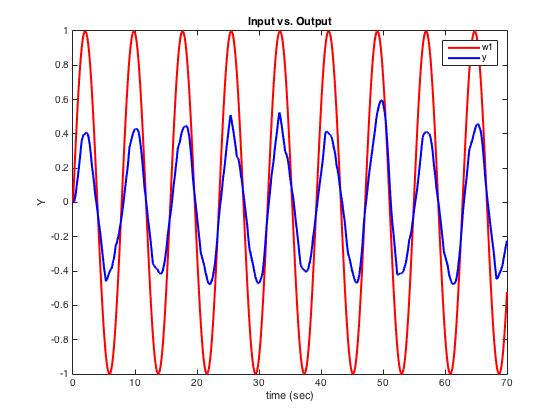}
		\caption{Simulation results for example 2.}
		\label{fig:NCSYpEpEX2Y}
	\end{figure}
\end{example} 	
\subsubsection{Calculating Passivity indices for Figure \ref{fig:ETNCSPS} (Event-triggering condition on Plant's Output)}

\begin{theorem}
	The networked control system given in figure \ref{fig:ETNCSPS} where $G_p$ and $G_c$ have passivity indices $\nu_p$, $\rho_p$, $\nu_c$, and $\rho_c$ and the triggering instance $t_k$ is determined by the triggering condition $||e_p(t)||_2^2>\delta||y_p(t)||_2^2$ and $\delta  \in(0,1]$, is passive from the inputs 	$\begin{bmatrix}
	w_1(t) \\ w_2(t)      
	\end{bmatrix}$ to the outputs
	$\begin{bmatrix}
	y_p(t) \\ y_c(t)      
	\end{bmatrix}$ with input passivity index $\epsilon_0$ and output passivity index $\delta_0$:
	
	$\dot{V}(t) \leq 
	\begin{bmatrix}
	w^T_1(t) & w^T_2(t)      
	\end{bmatrix}
	\begin{bmatrix}
	y_p(t) \\ y_c(t)      
	\end{bmatrix}
	- \epsilon_0 
	\begin{bmatrix}
	w^T_1(t) & w^T_2(t)      
	\end{bmatrix}
	\begin{bmatrix}
	w_1(t) \\ w_2(t)      
	\end{bmatrix}
	-\delta_0 
	\begin{bmatrix}
	y_p^T(t) & y_c^T(t)      
	\end{bmatrix}
	\begin{bmatrix}
	y_p(t) \\ y_c(t)      
	\end{bmatrix}$,
	
	where:
	\begin{align*}
		&\epsilon_0 < \min(\nu_p,\nu_c-|\nu_c|)\\
		&\delta_0 < \min(\beta(\nu_c)-\frac{\nu_c^2}{\nu_c-|\nu_c|-\epsilon_0},\rho_c+\nu_p-\frac{1}{4}-\frac{\nu_p^2}{\nu_p-\epsilon_0})
	\end{align*}

\end{theorem}

\begin{proof}

	We want to calculate the passivity indices for the setup in figure \ref{fig:ETNCSPS} from inputs $[w_1(t)~ w_2(t)]^T$ to outputs $[y_p(t)~ y_c(t)]^T$. We know that the setup is \emph{QSR-disspative} such that:
	
	\begin{align*}
		&\dot{V}(t) \leq w^T(t)Rw(t)+2w^T(t)Sy(t)+y^T(t)Qy(t)
	\end{align*}
	
	where $w(t)=\begin{bmatrix}
	w_1(t) \\ w_2(t)   
	\end{bmatrix}$,
	$y(t)=\begin{bmatrix}
	y_p(t) \\ y_c(t)   
	\end{bmatrix}$ and:\\
	
	$S=\begin{bmatrix}
	\frac{1}{2}I & \nu_pI\\
	-\nu_cI & \frac{1}{2}I      
	\end{bmatrix}$, $R=\begin{bmatrix}
	-\nu_pI & 0I\\
	0I & -(\nu_c-|\nu_c|)I      
	\end{bmatrix}$, and $Q=\begin{bmatrix}
	-\beta(\nu_c)I & 0I\\
	0I & -(\rho_c+\nu_p-\frac{1}{4})I\end{bmatrix}$
	where:\\
	
	\[\beta(\nu_c) =
	\begin{cases}
	\rho_p-\nu_c-\delta(1+\nu_c)      & \quad \text{if } \nu_c > 0\\
	\rho_p-\delta       & \quad \text{if } \nu_c = 0\\
	\rho_p+2\nu_c-\delta (1-3\nu_c)  & \quad \text{if } \nu_c < 0\\
	\end{cases}
	\]
	we need to show that: \\

	$\dot{V}(t) \leq w^T(t)Rw(t)+2w^T(t)Sy(t)+y^T(t)Qy(t)\\~~~~~~~~~~~~ \leq 
	\begin{bmatrix}
	w^T_1(t) & w^T_2(t)      
	\end{bmatrix}
	\begin{bmatrix}
	y_p(t) \\ y_c(t)      
	\end{bmatrix}
	- \epsilon_0 
	\begin{bmatrix}
	w^T_1(t) & w^T_2(t)      
	\end{bmatrix}
	\begin{bmatrix}
	w_1(t) \\ w_2(t)      
	\end{bmatrix}
	-\delta_0 
	\begin{bmatrix}
	y_p^T(t) & y_c^T(t)      
	\end{bmatrix}
	\begin{bmatrix}
	y_p(t) \\ y_c(t)      
	\end{bmatrix}$~~~(12)

	and calculate the passivity indices $\epsilon_0$, and $\delta_0$.\\
	
	Simplifying (12) and moving the terms to one side we have: \\
	
	$\dot{V}(t)\leq  (\epsilon_0-\nu_p)w_1^T(t)w_1(t)-(\nu_c-|\nu_c|-\epsilon_0) w_2^T(t)w_2(t)-(\beta(\nu_c)-\delta_0)y^T_p(t)y_p(t)\\~~~~~~~~~~~~-(\rho_c+\nu_p-\frac{1}{4}-\delta_0)y^T_c(t)y_c(t)+2\nu_pw^T_1(t)y_c(t)-2\nu_cw^T_2(t)y_p(t)\\~~~~~~~~~~~\leq 0 $.\\
	
	For this to hold we should have:\\
	
	$\begin{bmatrix}
	w^T_1(t) & y_c^T(t)      
	\end{bmatrix} M
	\begin{bmatrix}
	w_1(t) \\ y_c(t)      
	\end{bmatrix} + \begin{bmatrix}
	w^T_2(t) & y^T_p(t)      \end{bmatrix} N \begin{bmatrix}
	w_2(t) \\ y_p(t)      \end{bmatrix} \leq 0$,\\
	
	where  $M=\begin{bmatrix}
	(\epsilon_0-\nu_p) & \nu_p\\
	\nu_p & -(\rho_c+\nu_p-\frac{1}{4}-\delta_0)
	\end{bmatrix}$ and  $N=\begin{bmatrix}
	-(\nu_c-|\nu_c|-\epsilon_0) & -\nu_c\\
	-\nu_c&-(\beta(\nu_c)-\delta_0) \end{bmatrix}$.\\ 
	
	For matrices M and N to be negative semi-definite, they need to meet the following conditions:
	
	\begin{align*}
		\epsilon_0 < \nu_p\\
		\epsilon_0 <  \nu_c-|\nu_c|\\
		\delta_0 < \beta(\nu_c)\\
		\delta_0 < \rho_c+\nu_p-\frac{1}{4}\\
		\nu_p^2\leq -(\epsilon_0-\nu_p)(\rho_c+\nu_p-\frac{1}{4}-\delta_0)\\
		\nu_c^2\leq (\beta(\nu_c)-\delta_0)(\nu_c-|\nu_c|-\epsilon_0)
	\end{align*}
	
	which also prove the theorem.
\end{proof}
\subsubsection{Passivity for Figure \ref{fig:ETNCSPS} from $w_1 \to y_p$ (Event-triggering condition on the Plant's Output)}
\begin{theorem}	
	The networked control system given in figure \ref{fig:ETNCSPS} where $G_p$ and $G_c$ have passivity indices $\nu_p$, $\rho_p$, $\nu_c$, and $\rho_c$ and the triggering instance $t_k$ is explicitly determined by the condition $||e_p(t)||_2^2 > \delta ||y_p(t)||_2^2$ with $\delta  \in(0,1]$, and $w_2=0$, is passive from $w_1 \to y_p$ meaning:\\
	\begin{align*}
		\dot{V}(t) \leq w_1^T(t)y_p(t),
	\end{align*}
	
	if:
	\begin{align*}
		\beta(\nu_c) \geq 0\\
		\nu_p > 0\\
		\rho_c > \frac{1}{4}\\
	\end{align*}	 
\end{theorem}

\begin{proof}
	
	We know that the setup is \emph{QSR-disspative} such that:
	
	\begin{align*}
		&\dot{V}(t) \leq w^T(t)Rw(t)+2w^T(t)Sy(t)+y^T(t)Qy(t)
	\end{align*}
	
	where $w(t)=\begin{bmatrix}
	w_1(t) \\ w_2(t)   
	\end{bmatrix}$,
	$y(t)=\begin{bmatrix}
	y_p(t) \\ y_c(t)   
	\end{bmatrix}$ and:\\
	
	$S=\begin{bmatrix}
	\frac{1}{2}I & \nu_pI\\
	-\nu_cI & \frac{1}{2}I      
	\end{bmatrix}$, $R=\begin{bmatrix}
	-\nu_pI & 0I\\
	0I & -(\nu_c-|\nu_c|)I      
	\end{bmatrix}$, and $Q=\begin{bmatrix}
	-\beta(\nu_c)I & 0I\\
	0I & -(\rho_c+\nu_p-\frac{1}{4})I\end{bmatrix}$
	and:\\
	
	\[\beta(\nu_c) =
	\begin{cases}
	\rho_p-\nu_c-\delta(1+\nu_c)      & \quad \text{if } \nu_c > 0\\
	\rho_p-\delta       & \quad \text{if } \nu_c = 0\\
	\rho_p+2\nu_c-\delta(1-3\nu_c)  & \quad \text{if } \nu_c < 0\\
	\end{cases}
	\]
	
	Given that $w_2=0$, we need to show that \\

	$\dot{V}(t)	\leq  w_1^T(t)y_p(t)-\nu_pw_1^T(t)w_1(t)-\beta(\nu_c)y^T_p(t)y_p(t)-(\rho_c+\nu_p-\frac{1}{4})y^T_c(t)y_c(t)+2\nu_pw^T_1(t)y_c(t)\\~~~~~~~~~~~~ \leq w_1^T(t)y_p(t).~~~~~~~~~~~~~~~~~~~~~~~~~~~~~~~~~~~~~~~~~~~~~~~~~~~~~~~~~~~~~~~~~~~~~~~~~~~~~~~~~~~(13)\\$
	
	Simplifying (13) and moving the terms to one side we have: \\
	
	$\dot{V}(t)\leq  -\nu_pw_1^T(t)w_1(t)-\beta(\nu_c)y^T_p(t)y_p(t)-(\rho_c+\nu_p-\frac{1}{4})y^T_c(t)y_c(t)+2\nu_pw^T_1(t)y_c(t)\\~~~~~~~~~~~\leq 0 $\\
	
	We need to show that\\
	
	$-\beta(\nu_c)y^T_p(t)y_p(t) + \begin{bmatrix}
	w^T_1(t) & y_c^T(t)      
	\end{bmatrix} M
	\begin{bmatrix}
	w_1(t) \\ y_c(t)      
	\end{bmatrix} \leq 0$,
	
	where $M=\begin{bmatrix}
	-\nu_p & \nu_p\\
	\nu_p &-(\rho_c+\nu_p-\frac{1}{4})
	\end{bmatrix}$ 
	it is easy to see that the system is passive if:
	\begin{align*}
		\beta(\nu_c) \geq 0\\
		\nu_p > 0\\
		\rho_c > \frac{1}{4}
	\end{align*}	 
\end{proof}
\subsubsection{Passivity and Passivity indices for Figure \ref{fig:ETNCSPS} from $w_1 \to y_p$ (Event-triggering condition on Plant's Output)}
\begin{theorem}	
	The networked control system given in figure \ref{fig:ETNCSPS} where $G_p$ and $G_c$ have passivity indices $\nu_p$, $\rho_p$, $\nu_c$, and $\rho_c$ and the triggering instance $t_{k}$ is explicitly determined by the condition $||e_p(t)||_2^2>\delta ||y_p(t)||_2^2$ with $\delta  \in(0,1]$, and $w_2=0$, is passive from $w_1 \to y_p$ with input passivity index $\epsilon_0$ and output passivity index $\delta_0$ meaning:\\
	\begin{align*}
		\dot{V}(t) \leq w_1^T(t)y_p(t)- \epsilon_0 w^T_1(t)w_1(t) -\delta_0 y_p^T(t)y_p(t),
	\end{align*}
	
	where
	\begin{align*}
		\epsilon_0 < \min(\nu_p, \frac{\nu_p(\rho_c-\frac{1}{4})}{\rho_c+\nu_p-\frac{1}{4}})\\
		0 \leq \delta_0 \leq \beta(\nu_c)\\
	\end{align*}
	
\end{theorem}

\begin{proof}
	
	We know that the setup is \emph{QSR-disspative} such that:
	
	\begin{align*}
		&\dot{V}(t) \leq w^T(t)Rw(t)+2w^T(t)Sy(t)+y^T(t)Qy(t)
	\end{align*}
	
	where $w(t)=\begin{bmatrix}
	w_1(t) \\ w_2(t)   
	\end{bmatrix}$,
	$y(t)=\begin{bmatrix}
	y_p(t) \\ y_c(t)   
	\end{bmatrix}$ and:\\
	
	$S=\begin{bmatrix}
	\frac{1}{2}I & \nu_pI\\
	-\nu_cI & \frac{1}{2}I      
	\end{bmatrix}$, $R=\begin{bmatrix}
	-\nu_pI & 0I\\
	0I & -(\nu_c-|\nu_c|)I      
	\end{bmatrix}$, and $Q=\begin{bmatrix}
	-\beta(\nu_c)I & 0I\\
	0I & -(\rho_c+\nu_p-\frac{1}{4})I\end{bmatrix}$
	and:\\
	
	\[\beta(\nu_c) =
	\begin{cases}
	\rho_p-\nu_c-\delta(1+\nu_c)      & \quad \text{if } \nu_c > 0\\
	\rho_p-\delta       & \quad \text{if } \nu_c = 0\\
	\rho_p+2\nu_c-\delta(1-3\nu_c)  & \quad \text{if } \nu_c < 0\\
	\end{cases}
	\]
	
	Given that $w_2=0$, we need to show that \\
	
	$\dot{V}(t)	\leq  w_1^T(t)y_p(t)-\nu_pw_1^T(t)w_1(t)-\beta(\nu_c)y^T_p(t)y_p(t)-(\rho_c+\nu_p-\frac{1}{4})y^T_c(t)y_c(t)+2\nu_pw^T_1(t)y_c(t)\\~~~~~~~~~~~~~~ \leq w_1^T(t)y_p(t)- \epsilon_0 w^T_1(t)w_1(t) -\delta_0 y_p^T(t)y_p(t). $~~~~~~~~~~~~~~~~~~~~~~~~~~~~~~~~~~~~~~~~~~~~~~~~~~~~~~~~~~~~~~~~~~~~~~~~~~~~(14)\\
	
	Simplifying (14) and moving the terms to one side we have: \\
	
	$\dot{V}(t)\leq  (\epsilon_0-\nu_p)w_1^T(t)w_1(t)+(\delta_0-\beta(\nu_c))y^T_p(t)y_p(t)-(\rho_c+\nu_p-\frac{1}{4})y^T_c(t)y_c(t)+2\nu_pw^T_1(t)y_c(t)\\~~~~~~~~~~~\leq 0 $.\\
	
	We need to show that
	
	$(\delta_0-\beta(\nu_c))y^T_p(t)y_p(t) + \begin{bmatrix}
	w^T_1(t) & y_c^T(t)      
	\end{bmatrix} M
	\begin{bmatrix}
	w_1(t) \\ y_c(t)      
	\end{bmatrix} \leq 0$,
	
	where $M=\begin{bmatrix}
	\epsilon_0-\nu_p & \nu_p\\
	\nu_p &-(\rho_c+\nu_p-\frac{1}{4})
	\end{bmatrix}$ 
	it is easy to see that for the relation to hold, the followings should be met:
	\begin{align*}
		\beta(\nu_c) \geq \delta_0\\
		\nu_p > \epsilon_0\\
		\nu_p(\rho_c-\frac{1}{4}) \geq \epsilon_0(\rho_c+\nu_p-\frac{1}{4})\\
	\end{align*}	
	
	hence we can conclude that if $\epsilon_0$ and $\delta_0$ are determined as mentioned in the theorem then we have:
	
	$\dot{V}(t)	\leq w_1^T(t)y_p(t)- \epsilon_0 w^T_1(t)w_1(t) -\delta_0 y_p^T(t)y_p(t)$, \\ 
	
	for $\forall w_1$ and $\forall y_p$
\end{proof}

\begin{example}
	For the feedback interconnection given in example 2, we know that $\rho_p=1$, $\nu_p=0$, $\rho_c=0.5$ and $\nu_c=0.3$. By choosing $\delta = 0.30$, all conditions given in theorem 4 hold. Accordingly we can calculate that $\epsilon_0=0$, and $\delta_0 \leq 0.3$. Figure \ref{fig:PassivityYpEp} shows that the system is output passive with output passivity index $\delta_0=0.3$ and the relation $\int_{0}^{T}(w_1^T(t)y_p(t)-0.3y_p^T(t)y_p(t)) dt$ holds.
	
	\begin{figure}[!t]
		\centering
		\includegraphics[scale = .5]{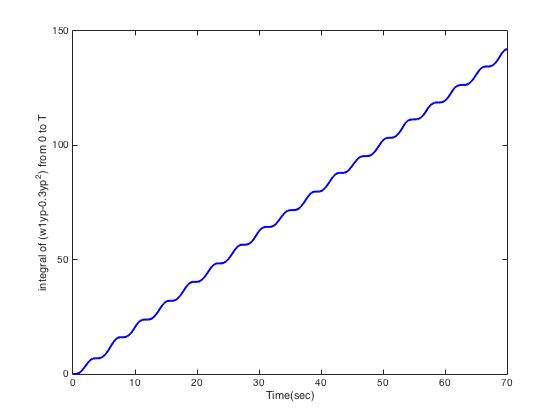}
		\caption{Simulation result for example 3 where $w_1=sin(\frac{5\pi}{2}t)$ .}
		\label{fig:PassivityYpEp}
	\end{figure}
\end{example}
\subsection{\emph{QSR-dissipativity} and Passivity Analysis - Event-triggering condition on the Controller's side - Figure \ref{fig:ETNCSCS})} 
\subsubsection{\emph{QSR-dissipativity} for NCS (Event-triggering condition on Controller's Output - Figure \ref{fig:ETNCSCS})}

\begin{theorem}
	Consider the feedback interconnection of two systems $G_p$ and $G_c$ in figure \ref{fig:ETNCSCS} with respective passivity indices of $\nu_p$, $\rho_p$, $\nu_c$ and $\rho_c$. If the event instance $t_k$ is explicitly determined by the triggering condition $||e_c(t)||_2^2>\delta||y_c(t)||_2^2$ where $\delta  \in(0,1]$, then the event-triggered networked control system is QSR-dissipative with respect to the inputs $w(t)=\begin{bmatrix}
	w_1(t) \\ w_2(t)   
	\end{bmatrix}$, and the outputs
	$y(t)=\begin{bmatrix}
	y_p(t) \\ y_c(t)   
	\end{bmatrix}$, and satisfies the relation:
	\begin{align*}
		&\dot{V}(t) \leq w^T(t)Rw(t)+2w^T(t)Sy(t)+y^T(t)Qy(t)
	\end{align*}
	where,\\
	
	$S=\begin{bmatrix}
	\frac{1}{2}I & \nu_pI\\
	-\nu_cI & \frac{1}{2}I      
	\end{bmatrix}$, $R=\begin{bmatrix}
	-(\nu_p-|\nu_p|)I & 0I\\
	0I & -\nu_cI      
	\end{bmatrix}$, and $Q=\begin{bmatrix}
	-(\rho_p+\nu_c-\frac{1}{4})I& 0I\\
	0I & -\beta(\nu_p)I\end{bmatrix}$
	and:\\
	
	\[\beta(\nu_p) =
	\begin{cases}
	\rho_c-\nu_p-\delta(1+\nu_p)      & \quad \text{if } \nu_p > 0\\
	\rho_c-\delta       & \quad \text{if } \nu_p = 0\\
	\rho_c+2\nu_p-\delta(1-3\nu_p)  & \quad \text{if } \nu_p < 0\\
	\end{cases}
	\]
	
	Additionally if $Q<0$ meaning $\beta(\nu_p)>0$ and $\rho_p+\nu_c-\frac{1}{4}>0$ then the interconnection is $L_2$-stable.
	
\end{theorem}

\begin{proof} 
	\label{thm:NCSYcEc}
	Given the systems $G_p$ and $G_c$ with passivity indices $\nu_p$, $\rho_p$, $\nu_c$ and $\rho_c$, there exists $V_p(t)$ and $V_c(t)$ such that:
	
	\begin{align*}
		& \dot{V}_p(t) \leq u^T_p(t)y_p(t)-\nu_pu^T_p(t)u_p(t)-\rho_p y^T_p(t)y_p(t) 
		\\&\dot{V}_c(t) \leq u^T_c(t)y_c(t)-\nu_c u^T_c(t)u_c(t)-\rho_c y^T_c(t)y_c(t)\\
	\end{align*}
	
	Additionally, according to the setup portrayed in figure \ref{fig:ETNCSCS}, the following relationships stand for $t \in [t_k,t_{k+1})$:
	
	\begin{align*}
		&  u_c(t)=w_2(t)+y_p(t)
		\\&e_c(t)=y_c(t)-y_c(t_k)
		\\&u_p(t)=w_1(t)-y_c(t_k)=w_1(t)+e_c(t)-y_c(t)
	\end{align*}
	
	and we design the triggering condition based on the following rule ($||e_c(t)||_2^2>\delta||y_c(t)||_2^2 $):
	
	\begin{align*}
		& \langle e_c,e_c\rangle>\delta \langle y_c,y_c\rangle, ~ 0 < \delta \leq 1
	\end{align*} 
	
	We consider the following storage function for the interconnection:
	
	\begin{align*}
		& V(t) = V_p(t) + V_c(t)
	\end{align*} 
	
	hence, we have:
	\begin{align*}
		& \dot{V}(t)=\dot{V}_p(t)+\dot{V}_c(t)\\&~~~~~~ \leq u^T_p(t)y_p(t)-\nu_pu^T_p(t)u_p(t)-\rho_p y^T_p(t)y_p(t)+u^T_c(t)y_c(t)-\nu_c u^T_c(t)u_c(t)-\rho_c y^T_c(t)y_c(t)
		\\
	\end{align*} 
	
	We know that $u_c(t)=w_2(t)+y_p(t)$, $u_p(t)=w_1(t)+e_c(t)-y_c(t)$, as a result for any $t \in [t_k,t_{k+1})$ we have:
	
	\begin{align*}
		&\dot{V}(t) \leq (w_1(t)+e_c(t)-y_c(t))^T y_p(t)-\nu_p(w_1(t)+e_c(t)-y_c(t))^T(w_1(t)+e_c(t)-y_c(t))-\rho_p y^T_p(t)y_p(t)\\&+(w_2(t)+y_p(t))^Ty_c(t)-\nu_c (w_2(t)+y_p(t))^T(w_2(t)+y_p(t))-\rho_c y^T_c(t)y_c(t)
		\\&= w_1^T(t)y_p(t)+w_2^T(t)y_c(t)-\nu_pw_1^T(t)w_1(t)-\nu_cw_2^T(t)w_2(t)-(\rho_p+\nu_c)y^T_p(t)y_p(t)-(\rho_c+\nu_p)y^T_c(t)y_c(t)\\&+2\nu_pw^T_1(t)y_c(t)-2\nu_cw^T_2(t)y_p(t)+2\nu_py^T_c(t)e_c(t)-2\nu_pw^T_1(t)e_c(t)+y^T_p(t)e_c(t)-\nu_pe_c^T(t)e_c(t).
	\end{align*} 
	
	Given that $-2\nu_pw^T_1(t)e_c(t) \leq |\nu_p|w_1^T(t)w_1(t)+|\nu_p|e_c^T(t)e_c(t)$ we have:
	
	$\dot{V}(t) \leq 2\begin{bmatrix}
	w^T_1(t) & w^T_2(t)      
	\end{bmatrix}\begin{bmatrix}
	\frac{1}{2} & \nu_p\\
	-\nu_c & \frac{1}{2}      
	\end{bmatrix}\begin{bmatrix}
	y_p(t) \\ y_c(t)      
	\end{bmatrix}+\begin{bmatrix}
	w^T_1(t) & w^T_2(t)      
	\end{bmatrix}\begin{bmatrix}
	|\nu_p|-\nu_p & 0\\
	0 & -\nu_c      
	\end{bmatrix}\begin{bmatrix}
	w_1(t) \\ w_2(t)      
	\end{bmatrix}\\
	+\begin{bmatrix}
	y^T_p(t) & y^T_c(t)      
	\end{bmatrix}\begin{bmatrix}
	-(\rho_p+\nu_c) & 0\\
	0 & -(\rho_c+\nu_p)     
	\end{bmatrix}\begin{bmatrix}
	y_p(t) \\ y_c(t)      
	\end{bmatrix}+2\nu_py^T_c(t)e_c(t)+|\nu_p|e_c^T(t)e_c(t)+y^T_p(t)e_c(t)\\-\nu_pe_c^T(t)e_c(t)$\\
	
	Additionally we have: $+y^T_p(t)e_c(t)=-( e_c(t) - \frac{1}{2} y_p(t))^2+ e_c^T(t)e_c(t) + 
	\frac{1}{4} y_p^T(t)y_p(t)$ , and $e_c^T(t)e_c(t) \leq \delta y_c^T(t)y_c(t)$ so we have:

	$\dot{V}(t) \leq 2\begin{bmatrix}
	w^T_1(t) & w^T_2(t)      
	\end{bmatrix}\begin{bmatrix}
	\frac{1}{2} & \nu_p\\
	-\nu_c & \frac{1}{2}      
	\end{bmatrix}\begin{bmatrix}
	y_p(t) \\ y_c(t)      
	\end{bmatrix}+\begin{bmatrix}
	w^T_1(t) & w^T_2(t)      
	\end{bmatrix}\begin{bmatrix}
	|\nu_p|-\nu_p & 0\\
	0 & -\nu_c      
	\end{bmatrix}\begin{bmatrix}
	w_1(t) \\ w_2(t)      
	\end{bmatrix}\\
	+\begin{bmatrix}
	y^T_p(t) & y^T_c(t)      
	\end{bmatrix}\begin{bmatrix}
	-(\rho_p+\nu_c-\frac{1}{4}) & 0\\
	0 & -(\rho_c+\nu_p-\delta)     
	\end{bmatrix}\begin{bmatrix}
	y_p(t) \\ y_c(t)      
	\end{bmatrix}+2\nu_py^T_c(t)e_c(t)+|\nu_p|e_c^T(t)e_c(t)-\nu_pe_c^T(t)e_c(t)~~~(15)$\\

	if $\nu_p > 0$ then we have: 
	
	$2\nu_py^T_c(t)e_c(t)-\nu_pe^T_c(t)e_c(t)=-\begin{bmatrix}
	e_c(t) & y_c(t)      
	\end{bmatrix} M \begin{bmatrix}
	e_c(t) \\ y_c(t)      
	\end{bmatrix} + 2 \nu_p y_c^T(t)y_c(t)
	$,   \\
	where $M=\begin{bmatrix}
	\nu_p & -\nu_p\\
	-\nu_p & 2\nu_p    
	\end{bmatrix}\geq0$, and $\nu_p e_c^T(t)e_c(t) \leq \delta \nu_p y_c^T(t)y_c(t)$ so we can simplify (15) further and have:
	
	$\dot{V}(t) \leq 2\begin{bmatrix}
	w^T_1(t) & w^T_2(t)      
	\end{bmatrix}\begin{bmatrix}
	\frac{1}{2} & \nu_p\\
	-\nu_c & \frac{1}{2}      
	\end{bmatrix}\begin{bmatrix}
	y_p(t) \\ y_c(t)      
	\end{bmatrix}+\begin{bmatrix}
	w^T_1(t) & w^T_2(t)      
	\end{bmatrix}\begin{bmatrix}
	|\nu_p|-\nu_p & 0\\
	0 & -\nu_c      
	\end{bmatrix}\begin{bmatrix}
	w_1(t) \\ w_2(t)      
	\end{bmatrix}\\
	+\begin{bmatrix}
	y^T_p(t) & y^T_c(t)      
	\end{bmatrix}\begin{bmatrix}
	-(\rho_p+\nu_c-\frac{1}{4})& 0\\
	0 & -(\rho_c-\nu_p-\delta(1+\nu_p))     
	\end{bmatrix}\begin{bmatrix}
	y_p(t) \\ y_c(t)   
	\end{bmatrix}$\\
	
	if $\nu_p = 0$ then we have: \\
	
	$\dot{V}(t) \leq 2\begin{bmatrix}
	w^T_1(t) & w^T_2(t)      
	\end{bmatrix}\begin{bmatrix}
	\frac{1}{2} & \nu_p\\
	-\nu_c & \frac{1}{2}      
	\end{bmatrix}\begin{bmatrix}
	y_p(t) \\ y_c(t)      
	\end{bmatrix}+\begin{bmatrix}
	w^T_1(t) & w^T_2(t)      
	\end{bmatrix}\begin{bmatrix}
	|\nu_p|-\nu_p & 0\\
	0 & -\nu_c      
	\end{bmatrix}\begin{bmatrix}
	w_1(t) \\ w_2(t)      
	\end{bmatrix}\\
	+\begin{bmatrix}
	y^T_p(t) & y^T_c(t)      
	\end{bmatrix}\begin{bmatrix}
	-(\rho_p+\nu_c-\frac{1}{4})& 0\\
	0 & -(\rho_c-\delta)     
	\end{bmatrix}\begin{bmatrix}
	y_p(t) \\ y_c(t)   
	\end{bmatrix}$\\
	
	if $\nu_p < 0$, given that $2\nu_py^T_c(t)e_c(t) \leq |\nu_p|y_c^T(t)y_c(t)+|\nu_p|e_c^T(t)e_c(t)$, $|\nu_p|e_c^T(t)e_c(t) \leq \delta|\nu_p| y_c^T(t)y_c(t)$ and $-\nu_p e_c^T(t)e_c(t) \leq \delta |\nu_p| y_c^T(t)y_c(t)$ then we have:\\ 
	
	$\dot{V}(t) \leq 2\begin{bmatrix}
	w^T_1(t) & w^T_2(t)      
	\end{bmatrix}\begin{bmatrix}
	\frac{1}{2} & \nu_p\\
	-\nu_c & \frac{1}{2}      
	\end{bmatrix}\begin{bmatrix}
	y_p(t) \\ y_c(t)      
	\end{bmatrix}+\begin{bmatrix}
	w^T_1(t) & w^T_2(t)      
	\end{bmatrix}\begin{bmatrix}
	|\nu_p|-\nu_p & 0\\
	0 & -\nu_c      
	\end{bmatrix}\begin{bmatrix}
	w_1(t) \\ w_2(t)      
	\end{bmatrix}\\
	+\begin{bmatrix}
	y^T_p(t) & y^T_c(t)      
	\end{bmatrix}\begin{bmatrix}
	-(\rho_p+\nu_c-\frac{1}{4})& 0\\
	0 & -(\rho_c+2\nu_p-\delta(1-3\nu_p))     
	\end{bmatrix}\begin{bmatrix}
	y_p(t) \\ y_c(t)   
	\end{bmatrix}$
\end{proof}

Hence, we have shown that:

\begin{align*}
	&\dot{V}(t) \leq w^T(t)Rw(t)+2w^T(t)Sy(t)+y^T(t)Qy(t)
\end{align*}

where $w(t)=\begin{bmatrix}
w_1(t) \\ w_2(t)   
\end{bmatrix}$,
$y(t)=\begin{bmatrix}
y_p(t) \\ y_c(t)   
\end{bmatrix}$ and:\\

$S=\begin{bmatrix}
\frac{1}{2}I & \nu_pI\\
-\nu_cI & \frac{1}{2}I      
\end{bmatrix}$, $R=\begin{bmatrix}
-(\nu_p-|\nu_p|)I & 0I\\
0I & -\nu_cI      
\end{bmatrix}$, and $Q=\begin{bmatrix}
-(\rho_p+\nu_c-\frac{1}{4})I& 0I\\
0I & -\beta(\nu_p)I\end{bmatrix}$
and:\\

\[\beta(\nu_p) =
\begin{cases}
\rho_c-\nu_p-\delta(1+\nu_p)      & \quad \text{if } \nu_p > 0\\
\rho_c-\delta       & \quad \text{if } \nu_p = 0\\
\rho_c+2\nu_p-\delta(1-3\nu_p)  & \quad \text{if } \nu_p < 0\\
\end{cases}
\]
\subsubsection{Simulation Examples for \emph{QSR-dissipativity} for NCS (Event-triggering condition on the Controller's Output - Figure \ref{fig:ETNCSCS}) }
 
 \begin{example}
 	The plant is defined by the following model:	
 	
 	\begin{align*} 
 	&\dot{x}_{p1}(t)=-3x^3_{p1}(t)+x_{p1}(t)x_{p2}(t)\\
 	&\dot{x}_{p2}(t)=0.2x_{p2}(t)+2u_p(t)\\
 	&y_p(t)=x_{p2}(t),
 	\end{align*}
 	
 	where $\rho_p=-0.2$ and $\nu_p=0$. And the model for the controller is the following:
 	
 	\begin{align*} 
 	&\dot{x}_{c}(t)=-3x_{c}(t)+u_{c}(t)\\
 	&y_c(t)=7x_{c}(t)+u_c(t),
 	\end{align*}
 	
 	where $\rho_c=0.3$ and $\nu_c=1$. The interconnection satisfies $\rho_p+\nu_c-\frac{1}{4}>0$, and by picking $\delta=0.2$, $\rho_c-\delta >0$ is satisfied, $w1=0$, and $w2=0$, the simulation results are given in: figures  \ref{fig:NCSYcEcEX1XP1}, \ref{fig:NCSYcEcEX1XP2}, \ref{fig:NCSYcEcEX1Y}.
 	\begin{figure}[!t]
 		\centering
 		\includegraphics[scale = .5]{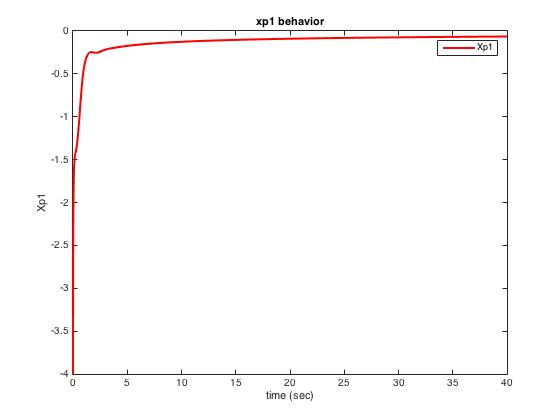}
 		\caption{Simulation results for example 4.}
 		\label{fig:NCSYcEcEX1XP1}
 	\end{figure}
 	\begin{figure}[!t]
 		\centering
 		\includegraphics[scale = .5]{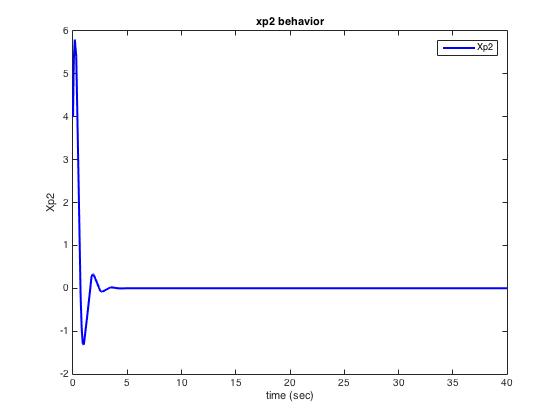}
 		\caption{Simulation results for example 4.}
 		\label{fig:NCSYcEcEX1XP2}
 	\end{figure}
 	\begin{figure}[!t]
 		\centering
 		\includegraphics[scale = .5]{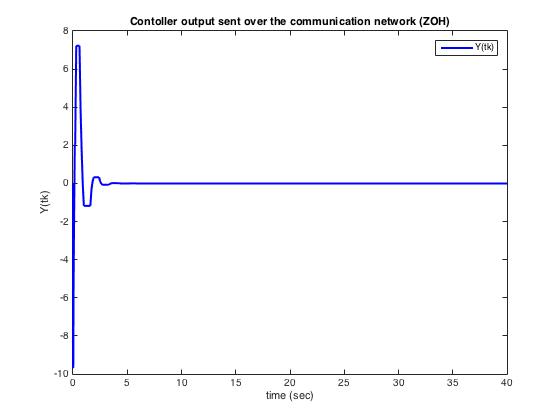}
 		\caption{Simulation results for example 4.}
 		\label{fig:NCSYcEcEX1Y}
 	\end{figure}
 \end{example}
 
 \begin{example}
 	
 	The plant in figure \ref{fig:NCSYcEcEX2} is defined by the following model:	
 	
 	\begin{figure}[!t]
 		\centering
 		\includegraphics[scale = .5]{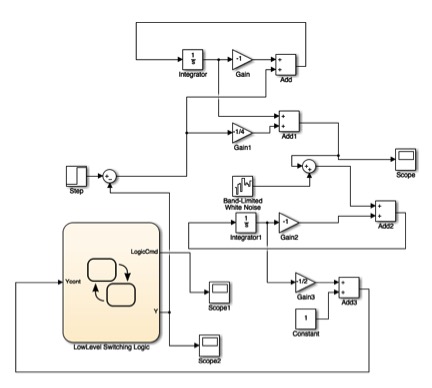}
 		\caption{Simulink model for example 5.}
 		\label{fig:NCSYcEcEX2}
 	\end{figure}

 	\begin{align*} 
 	&\dot{x}_{p1}(t)=-x_{p1}(t)+u_p(t)\\
 	&y_p(t)=x_{p1}(t)-0.25u_p(t),\\
 	\end{align*}
 	
 	where $\rho_p=2$ and $\nu_p=-.37$. And the model for the controller is following:
 	
 	\begin{align*} 
 	&\dot{x}_{c1}(t)=-x_{c1}(t)+u_c(t)\\
 	&y_c(t)=-0.5x_{c1}(t)+1,\\
 	\end{align*}
 	
 	where $\rho_c=1$ and $\nu_c=0.5$. The interconnection satisfies $\rho_p+\nu_c-\frac{1}{4}>0$, and by picking $\delta=0.1$, $\rho_c+2\nu_p-\delta(1-3\nu_p)$ is satisfied, $w1$ is a step signal, and $w2$ is white noise with power $0.08$, the simulation results are given in: figures  \ref{fig:NCSYcEcEX2Ytk}, \ref{fig:NCSYcEcEX2Y}.
 	
 	\begin{figure}[!t]
 		\centering
 		\includegraphics[scale = .5]{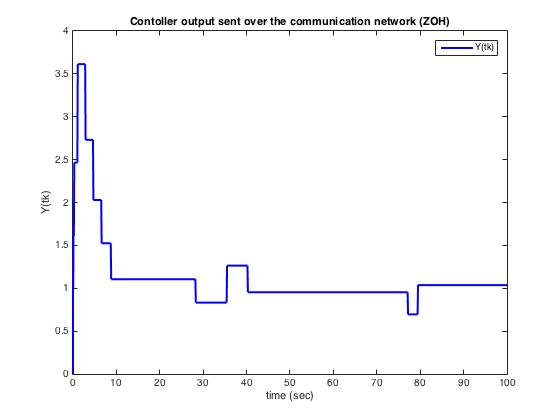}
 		\caption{Simulation results for example 5.}
 		\label{fig:NCSYcEcEX2Ytk}
 	\end{figure}
 	\begin{figure}[!t]
 		\centering
 		\includegraphics[scale = .5]{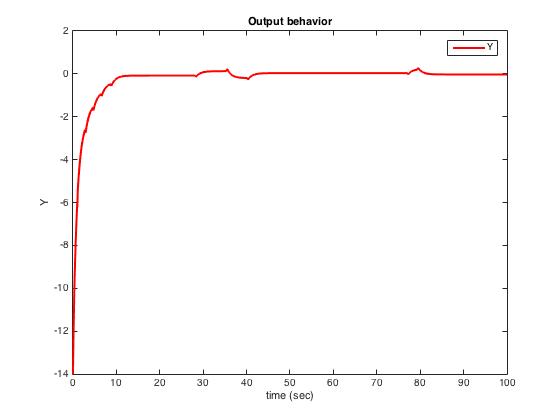}
 		\caption{Simulation results for example 5.}
 		\label{fig:NCSYcEcEX2Y}
 	\end{figure}
 \end{example}
 \subsubsection{Calculating Passivity indices for Figure \ref{fig:ETNCSCS} (Event-triggering condition on the Controller's Output)}
 
 \begin{theorem}
 	The networked control system given in figure \ref{fig:ETNCSCS} where $G_p$ and $G_c$ have passivity indices $\nu_p$, $\rho_p$, $\nu_c$, and $\rho_c$ and the triggering instance $t_k$ is determined by the triggering condition $||e_c(t)||_2^2>\delta||y_c(t)||_2^2$ and $\delta  \in(0,1]$, is passive from the inputs 	$\begin{bmatrix}
 	w_1(t) \\ w_2(t)      
 	\end{bmatrix}$ to the outputs
 	$\begin{bmatrix}
 	y_p(t) \\ y_c(t)      
 	\end{bmatrix}$ with input passivity index $\epsilon_0$ and output passivity index $\delta_0$:
 	
 	$\dot{V}(t) \leq 
 	\begin{bmatrix}
 	w^T_1(t) & w^T_2(t)      
 	\end{bmatrix}
 	\begin{bmatrix}
 	y_p(t) \\ y_c(t)      
 	\end{bmatrix}
 	- \epsilon_0 
 	\begin{bmatrix}
 	w^T_1(t) & w^T_2(t)      
 	\end{bmatrix}
 	\begin{bmatrix}
 	w_1(t) \\ w_2(t)      
 	\end{bmatrix}
 	-\delta_0 
 	\begin{bmatrix}
 	y_p^T(t) & y_c^T(t)      
 	\end{bmatrix}
 	\begin{bmatrix}
 	y_p(t) \\ y_c(t)      
 	\end{bmatrix}$
 	
 	where:
 	\begin{align*}
 	&\epsilon_0 <\min(\nu_c,\nu_p-|\nu_p|)\\
 	&\delta_0 <\min(\beta(\nu_p)-\frac{\nu_p^2}{\nu_p-|\nu_p|-\epsilon_0},\rho_p+\nu_c-\frac{1}{4}-\frac{\nu_c^2}{\nu_c-\epsilon_0})
 	\end{align*}

 \end{theorem}
 
 \begin{proof}
 	We want to calculate the passivity indices for the setup in figure \ref{fig:ETNCSCS} from inputs $[w_1(t)~ w_2(t)]^T$ to outputs $[y_p(t)~ y_c(t)]^T$. We know that the setup is \emph{QSR-disspative} such that:
 	
 	\begin{align*}
 	&\dot{V}(t) \leq w^T(t)Rw(t)+2w^T(t)Sy(t)+y^T(t)Qy(t)
 	\end{align*}
 	
 	where $w(t)=\begin{bmatrix}
 	w_1(t) \\ w_2(t)   
 	\end{bmatrix}$,
 	$y(t)=\begin{bmatrix}
 	y_p(t) \\ y_c(t)   
 	\end{bmatrix}$ and:\\
 	
 	$S=\begin{bmatrix}
 	\frac{1}{2}I & \nu_pI\\
 	-\nu_cI & \frac{1}{2}I      
 	\end{bmatrix}$, $R=\begin{bmatrix}
 	-(\nu_p-|\nu_p|)I & 0I\\
 	0I & -\nu_cI      
 	\end{bmatrix}$, and $Q=\begin{bmatrix}
 	-(\rho_p+\nu_c-\frac{1}{4})I& 0I\\
 	0I & -\beta(\nu_p)I\end{bmatrix}$
 	and:\\
 	
 	\[\beta(\nu_p) =
 	\begin{cases}
 	\rho_c-\nu_p-\delta(1+\nu_p)      & \quad \text{if } \nu_p > 0\\
 	\rho_c-\delta       & \quad \text{if } \nu_p = 0\\
 	\rho_c+2\nu_p-\delta(1-3\nu_p)  & \quad \text{if } \nu_p < 0\\
 	\end{cases}
 	\]
 	
 	we need to show that

 	$\dot{V}(t) \leq w^T(t)Rw(t)+2w^T(t)Sy(t)+y^T(t)Qy(t)\\~~~~~~~~~~~~ \leq 
 	\begin{bmatrix}
 	w^T_1(t) & w^T_2(t)      
 	\end{bmatrix}
 	\begin{bmatrix}
 	y_p(t) \\ y_c(t)      
 	\end{bmatrix}
 	- \epsilon_0 
 	\begin{bmatrix}
 	w^T_1(t) & w^T_2(t)      
 	\end{bmatrix}
 	\begin{bmatrix}
 	w_1(t) \\ w_2(t)      
 	\end{bmatrix}
 	-\delta_0 
 	\begin{bmatrix}
 	y_p^T(t) & y_c^T(t)      
 	\end{bmatrix}
 	\begin{bmatrix}
 	y_p(t) \\ y_c(t)      
 	\end{bmatrix}$~~~(16)

 	and calculate the passivity indices $\epsilon_0$, and $\delta_0$.\\
 	
 	Simplifying (16) and moving the terms to one side we have: \\
 	
 	$\dot{V}(t)\leq  (\epsilon_0-\nu_c)w_2^T(t)w_2(t)-(\nu_p-|\nu_p|-\epsilon_0) w_1^T(t)w_1(t)-(\beta(\nu_p)-\delta_0)y^T_c(t)y_c(t)\\~~~~~~~~~~~~-(\rho_p+\nu_c-\frac{1}{4}-\delta_0)y^T_p(t)y_p(t)+2\nu_pw^T_1(t)y_c(t)-2\nu_cw^T_2(t)y_p(t)\\~~~~~~~~~~~\leq 0 $
 	
 	As a result we have to show:\\
 	
 	$\begin{bmatrix}
 	w^T_1(t) & y_c^T(t)      
 	\end{bmatrix} M
 	\begin{bmatrix}
 	w_1(t) \\ y_c(t)      
 	\end{bmatrix} + \begin{bmatrix}
 	w^T_2(t) & y^T_p(t)      \end{bmatrix} N \begin{bmatrix}
 	w_2(t) \\ y_p(t)      \end{bmatrix} \leq 0$,

 	where  $M=\begin{bmatrix}
 	-(\nu_p-|\nu_p|-\epsilon_0) & \nu_p\\
 	\nu_p &-(\beta(\nu_p)-\delta_0)
 	\end{bmatrix}$ and  $N=\begin{bmatrix}
 	(\epsilon_0-\nu_c) & -\nu_c\\
 	-\nu_c& -(\rho_p+\nu_c-\frac{1}{4}-\delta_0) \end{bmatrix}$.\\\\ 
 	For matrices M and N to be negative semi-definite, they need to meet the following conditions:
 	
 	\begin{align*}
 	\epsilon_0 < \nu_c\\
 	\epsilon_0 <  \nu_p-|\nu_p|\\
 	\delta_0 < \beta(\nu_p)\\
 	\delta_0 < \rho_p+\nu_c-\frac{1}{4}\\
 	\nu_p^2\leq (\beta(\nu_p)-\delta_0)(\nu_p-|\nu_p|-\epsilon_0)\\
 	\nu_c^2\leq -(\epsilon_0-\nu_c)(\rho_p+\nu_c-\frac{1}{4}-\delta_0)
 	\end{align*}
 	
 	which also prove the theorem.
 \end{proof}
 \subsubsection{Passivity for Figure \ref{fig:ETNCSCS} from $w_1 \to y_p$ (Event-triggering condition on the Controller's Output)}
 \begin{theorem}	
 	The networked control system given in figure \ref{fig:ETNCSCS} where $G_p$ and $G_c$ have passivity indices $\nu_p$, $\rho_p$, $\nu_c$, and $\rho_c$ and the triggering instance $t_{k}$ is explicitly determined by the condition $||e_c(t)||_2^2>\delta ||y_c(t)||_2^2$ with $\delta  \in(0,1]$, and $w_2=0$, is passive from $w_1 \to y_p$ meaning:\\
 	\begin{align*}
 	\dot{V}(t) \leq w_1^T(t)y_p(t),
 	\end{align*}
 	
 	if:
 	\begin{align*}
 	\rho_c \geq \delta\\
 	\nu_p = 0\\
 	\rho_p + \nu_c \geq \frac{1}{4}\\
 	\end{align*}	 
 \end{theorem}
 
 \begin{proof}
 	
 	We know that the setup is \emph{QSR-disspative} such that:
 	
 	Hence, we have shown that:
 	
 	\begin{align*}
 	&\dot{V}(t) \leq w^T(t)Rw(t)+2w^T(t)Sy(t)+y^T(t)Qy(t)
 	\end{align*}
 	
 	where $w(t)=\begin{bmatrix}
 	w_1(t) \\ w_2(t)   
 	\end{bmatrix}$,
 	$y(t)=\begin{bmatrix}
 	y_p(t) \\ y_c(t)   
 	\end{bmatrix}$ and:\\
 	
 	$S=\begin{bmatrix}
 	\frac{1}{2}I & \nu_pI\\
 	-\nu_cI & \frac{1}{2}I      
 	\end{bmatrix}$, $R=\begin{bmatrix}
 	-(\nu_p-|\nu_p|)I & 0I\\
 	0I & -\nu_cI      
 	\end{bmatrix}$, and $Q=\begin{bmatrix}
 	-(\rho_p+\nu_c-\frac{1}{4})I& 0I\\
 	0I & -\beta(\nu_p)I\end{bmatrix}$
 	and:\\
 	
 	\[\beta(\nu_p) =
 	\begin{cases}
 	\rho_c-\nu_p-\delta(1+\nu_p)      & \quad \text{if } \nu_p > 0\\
 	\rho_c-\delta       & \quad \text{if } \nu_p = 0\\
 	\rho_c+2\nu_p-\delta(1-3\nu_p)  & \quad \text{if } \nu_p < 0\\
 	\end{cases}
 	\]
 	
 	Given that $w_2=0$, we need to show that \\
 	
 	$\dot{V}(t)	\leq 
 	w_1^T(t)y_p(t)-(\nu_p-|\nu_p|) w_1^T(t)w_1(t)-\beta(\nu_p)y^T_c(t)y_c(t)-(\rho_p+\nu_c-\frac{1}{4})y^T_p(t)y_p(t)+2\nu_pw^T_1(t)y_c(t)\\~~~~~~~~~~~~ \leq w_1^T(t)y_p(t)$.~~~~~~~~~~~~~~~~~~~~~~~~~~~~~~~~~~~~~~~~~~~~~~~~~~~~~~~~~~~~~~~~~~~~~~~~~~~~~~~~~~~~~(17)\\
 	
 	Simplifying (17) and moving the terms to one side, and using the relation $2\nu_pw_1^T(t)y_c(t) \leq |\nu_p|y_c^T(t)y_c(t)+|\nu_p|w_1^T(t)w_1(t)$: \\
 	
 	$~~~~~\dot{V}(t)	\leq 
 	-(\nu_p-|\nu_p|) w_1^T(t)w_1(t)-\beta(\nu_p)y^T_c(t)y_c(t)-(\rho_p+\nu_c-\frac{1}{4})y^T_p(t)y_p(t)+2\nu_pw^T_1(t)y_c(t)\\~~~~~~~~~~~~~~~~~	\leq 
 	-(\nu_p-2|\nu_p|) w_1^T(t)w_1(t)-(\beta(\nu_p)-|\nu_p|)y^T_c(t)y_c(t)-(\rho_p+\nu_c-\frac{1}{4})y^T_p(t)y_p(t)\\~~~~~~~~~~~~~~~~~\leq 0 $\\

 	it is easy to see that the system is passive if:
 	\begin{align*}
 	\beta(\nu_p) \geq 0\\
 	\nu_p = 0\\
 	\rho_p + \nu_c \geq \frac{1}{4}
 	\end{align*}
 \end{proof}
 \begin{example}
 	For the feedback interconnection given in example 4, we know that $\rho_p=-0.2$, $\nu_p=0$, $\rho_c=0.3$ and $\nu_c=1$. By choosing $\delta = 0.20$, all conditions given in theorem 7 hold, and accordingly we can say that the networked control system is passive. Figure \ref{fig:PassivityYcEc} shows that the relation $\int_{0}^{T}(w_1^T(t)y_p(t)) dt$ holds.
 	
 	\begin{figure}[!t]
 		\centering
 		\includegraphics[scale = .5]{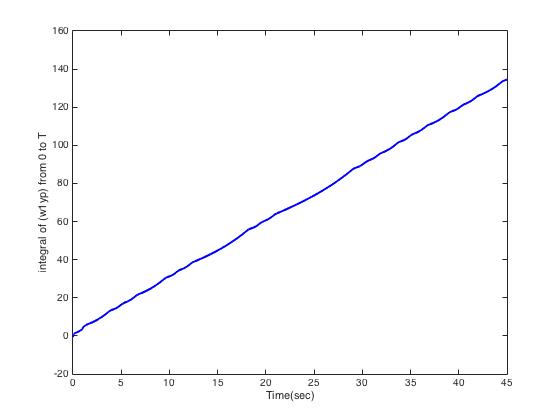}
 		\caption{Simulation result for example 6 where $w_1$ is a step function.}
 		\label{fig:PassivityYcEc}
 	\end{figure}
 \end{example}	

\subsubsection{Passivity and Passivity indices for Figure \ref{fig:ETNCSCS} from $w_1 \to y_p$ (Event-triggering condition on the Controller's Output) }
\begin{theorem}	
	The networked control system given in figure \ref{fig:ETNCSCS} where $G_p$ and $G_c$ have passivity indices $\nu_p$, $\rho_p$, $\nu_c$, and $\rho_c$ and the triggering instance $t_{k}$ is explicitly determined by the condition $||e_c(t)||_2^2>\delta ||y_c(t)||_2^2$ with $\delta  \in(0,1]$, and $w_2=0$, is passive from $w_1 \to y_p$ with input passivity index $\epsilon_0$ and output passivity index $\delta_0$ meaning:\\
	\begin{align*}
	\dot{V}(t) \leq w_1^T(t)y_p(t)- \epsilon_0 w^T_1(t)w_1(t) -\delta_0 y_p^T(t)y_p(t), 
	\end{align*}
	
	where
	\begin{align*}
	\epsilon_0 <\min(\nu_p,\nu_p-|\nu_p|-\frac{\nu_p^2}{\beta(\nu_p)}) \\
	0\leq \delta_0 \leq 	\rho_p+\nu_c-\frac{1}{4}\\\\
	\end{align*}	 
\end{theorem}

\begin{proof}
	
	We know that the setup is \emph{QSR-disspative} such that:

	\begin{align*}
	&\dot{V}(t) \leq w^T(t)Rw(t)+2w^T(t)Sy(t)+y^T(t)Qy(t),
	\end{align*}
	
	where $w(t)=\begin{bmatrix}
	w_1(t) \\ w_2(t)   
	\end{bmatrix}$,
	$y(t)=\begin{bmatrix}
	y_p(t) \\ y_c(t)   
	\end{bmatrix}$ and:\\
	
	$S=\begin{bmatrix}
	\frac{1}{2}I & \nu_pI\\
	-\nu_cI & \frac{1}{2}I      
	\end{bmatrix}$, $R=\begin{bmatrix}
	-(\nu_p-|\nu_p|)I & 0I\\
	0I & -\nu_cI      
	\end{bmatrix}$, and $Q=\begin{bmatrix}
	-(\rho_p+\nu_c-\frac{1}{4})I& 0I\\
	0I & -\beta(\nu_p)I\end{bmatrix}$
	and:\\
	
	\[\beta(\nu_p) =
	\begin{cases}
	\rho_c-\nu_p-\delta(1+\nu_p)      & \quad \text{if } \nu_p > 0\\
	\rho_c-\delta       & \quad \text{if } \nu_p = 0\\
	\rho_c+2\nu_p-\delta(1-3\nu_p)  & \quad \text{if } \nu_p < 0\\
	\end{cases}
	\]
	
	Given that $w_2=0$, we need to show that \\
	
	$\dot{V}(t)	\leq  w_1^T(t)y_p(t)-(\nu_p-|\nu_p|) w_1^T(t)w_1(t)-\beta(\nu_p)y^T_c(t)y_c(t)-(\rho_p+\nu_c-\frac{1}{4})y^T_p(t)y_p(t)+2\nu_pw^T_1(t)y_c(t)\\~~~~~~~~~~~~~~ \leq w_1^T(t)y_p(t)- \epsilon_0 w^T_1(t)w_1(t) -\delta_0 y_p^T(t)y_p(t) $.~~~~~~~~~~~~~~~~~~~~~~~~~~~~~~~~~~~~~~~~~~~~~~~~~~~~~~~~~~~~~~~~~~~~~~~~~~~~(18)\\
	
	Simplifying (18) and moving the terms to one side we have: \\
	
	$\dot{V}(t)\leq  (\epsilon_0-\nu_p+|\nu_p|) w_1^T(t)w_1(t)-\beta(\nu_p)y^T_c(t)y_c(t)-(\rho_p+\nu_c-\frac{1}{4}-\delta_0)y^T_p(t)y_p(t)+2\nu_pw^T_1(t)y_c(t)\\~~~~~~~~~~~\leq 0 $.\\
	
	We need to show that
	
	$-(\rho_p+\nu_c-\frac{1}{4}-\delta_0)y^T_p(t)y_p(t) + \begin{bmatrix}
	w^T_1(t) & y_c^T(t)      
	\end{bmatrix} M
	\begin{bmatrix}
	w_1(t) \\ y_c(t)      
	\end{bmatrix} \leq 0$,
	
	where $M=\begin{bmatrix}
	(\epsilon_0-\nu_p+|\nu_p|) & \nu_p\\
	\nu_p &-\beta(\nu_p)
	\end{bmatrix}$ 
	it is easy to see that for the relation to hold, the followings should be met:
	\begin{align*}
	\beta(\nu_p) > 0\\
	\nu_p-|\nu_p| > \epsilon_0\\
	\rho_p+\nu_c-\frac{1}{4}\geq \delta_0\\
	(\nu_p-|\nu_p|-\epsilon_0)\beta(\nu_p) \geq \nu_p^2\\
	\end{align*}	
	
	hence we can conclude that if $\epsilon_0$ and $\delta_0$ are determined as mentioned in the theorem then:
	
	$\dot{V}(t)	\leq w_1^T(t)y_p(t)- \epsilon_0 w^T_1(t)w_1(t) -\delta_0 y_p^T(t)y_p(t) $ \\	
	for $\forall w_1$ and $\forall y_p$
\end{proof}
\subsection{\emph{QSR-disspativity} and Passivity Analysis - Event-triggering condition on the Plant's Side and Controller's Side - Figure \ref{fig:ETNCSPSCS})}
\subsubsection{\emph{QSR-dissipativity} for NCS (Event-triggering condition on the Plant's Output and Controller's Output - Figure \ref{fig:ETNCSPSCS})}

\begin{theorem}
	Consider the feedback interconnection of two systems $G_p$ and $G_c$ in figure \ref{fig:ETNCSPSCS} with respective passivity indices of $\nu_p$, $\rho_p$, $\nu_c$ and $\rho_c$. If the event instances $t_{p_k}$ and $t_{c_k}$ are explicitly determined by the conditions $||e_p(t)||_2^2>\delta_p ||y_p(t)||_2^2$ with $\delta_p  \in(0,1]$, and $||e_c(t)||_2^2>\delta_c ||y_c(t)||_2^2$ with $\delta_c  \in(0,1]$, then the event-triggered networked control system is \emph{QSR-dissipative} with respect to the inputs $w(t)=\begin{bmatrix}
	w_1(t) \\ w_2(t)   
	\end{bmatrix}$, and the outputs
	$y(t)=\begin{bmatrix}
	y_p(t) \\ y_c(t)   
	\end{bmatrix}$, and satisfies the relation:
	\begin{align*}
	&\dot{V}(t) \leq w^T(t)Rw(t)+ 2w^T(t)Sy(t)+y^T(t)Qy(t)
	\end{align*}
	where:\\
	
	$S=\begin{bmatrix}
	\frac{1}{2}I & \nu_pI\\
	-\nu_cI & \frac{1}{2}I      
	\end{bmatrix}$, $R=\begin{bmatrix}
	-(\nu_p-|\nu_p|)I & 0I\\
	0I & -(\nu_c-|\nu_c|)I      
	\end{bmatrix}$, and $Q=\begin{bmatrix}
	-(\beta(\nu_c)-\frac{1}{4})I & 0I\\
	0I & -(\beta(\nu_p)-\frac{1}{4})I\end{bmatrix}$
	and:\\
	
	\[\beta(\nu_p) =
	\begin{cases}
	\rho_c-\nu_p-\delta_c(1+\nu_p)      & \quad \text{if } \nu_p > 0\\
	\rho_c-\delta_c       & \quad \text{if } \nu_p = 0\\
	\rho_c+2\nu_p-\delta_c(1-3\nu_p)  & \quad \text{if } \nu_p < 0\\
	\end{cases}
	\]
	
	\[\beta(\nu_c) =
	\begin{cases}
	\rho_p-\nu_c-\delta_p(1+\nu_c)      & \quad \text{if } \nu_c > 0\\
	\rho_p-\delta_p       & \quad \text{if } \nu_c = 0\\
	\rho_p+2\nu_c-\delta_p(1-3\nu_c)  & \quad \text{if } \nu_c < 0\\
	\end{cases}
	\]
	
	Additionally if $Q<0$ meaning $\beta(\nu_c)>\frac{1}{4}$ and $\beta(\nu_p)>\frac{1}{4}$ then the interconnection is $L_2$-stable.
	
\end{theorem}

\begin{proof} 
	\label{thm:NCSYpEpYcEc}
	
	Given systems $G_p$ and $G_c$ with passivity indices $\nu_p$, $\rho_p$, $\nu_c$ and $\rho_c$, there exists $V_p(t)$ and $V_c(t)$ such that:
	
	\begin{align*}
	& \dot{V}_p(t) \leq u^T_p(t)y_p(t)-\nu_pu^T_p(t)u_p(t)-\rho_p y^T_p(t)y_p(t) 
	\\&\dot{V}_c(t) \leq u^T_c(t)y_c(t)-\nu_c u^T_c(t)u_c(t)-\rho_c y^T_c(t)y_c(t)\\
	\end{align*}
	
	Additionally, according to the setup portrayed in figure \ref{fig:ETNCSPSCS}, the following relationships stand for $t\in [t_{p_k},t_{p_{k+1}}) \cup [t_{c_k},t_{c_{k+1}})$ :
	
	\begin{align*}
	& u_p(t)=w_1(t)-y_c(t_k)\\
	& u_c(t)=w_2(t)+y_p(t_k)\\
	&e_p(t)=y_p(t)-y_p(t_k)\\
	&e_c(t)=y_c(t)-y_c(t_k)\\
	&u_c(t)=w_2(t)+y_p(t)-e_p(t)\\
	&u_p(t)=w_1(t)+e_c(t)-y_c(t)
	\end{align*}
	and we design the triggering conditions based on the following rules ($||e_p(t)||_2^2>\delta_p||y_p(t)||_2^2$, $||e_c(t)||_2^2> \delta_c||y_c(t)||_2^2$):
	
	\begin{align*}
	& \langle e_p,e_p\rangle>\delta_p \langle y_p,y_p\rangle, ~ 0 < \delta_p \leq 1\\
	& \langle e_c,e_c\rangle>\delta_c \langle y_c,y_c\rangle, ~ 0 < \delta_c \leq 1
	\end{align*} 
	We consider the following storage function for the interconnection:
	
	\begin{align*}
	& V(t) = V_p(t) + V_c(t),
	\end{align*} 
	
	hence, we have:
	\begin{align*}
	& \dot{V}(t)=\dot{V}_p(t)+\dot{V}_c(t)\\&~~~~~~ \leq u^T_p(t)y_p(t)-\nu_pu^T_p(t)u_p(t)-\rho_p y^T_p(t)y_p(t)+u^T_c(t)y_c(t)-\nu_c u^T_c(t)u_c(t)-\rho_c y^T_c(t)y_c(t).
	\\
	\end{align*} 
	
	We know that $u_p(t)=w_1(t)+e_c(t)-y_c(t)$, $u_c(t)=w_2(t)+y_p(t)-e_p(t)$, as a result for any $t \in [t_{p_k},t_{P_{k+1}})$ and $t \in [t_{c_k},t_{c_{k+1}})$ we have:
	
	\begin{align*}
	&\dot{V}(t) \leq (w_1(t)+e_c(t)-y_c(t))^T y_p(t)-\nu_p(w_1(t)+e_c(t)-y_c(t))^T(w_1(t)+e_c(t)-y_c(t))-\rho_p y^T_p(t)y_p(t)\\&+(w_2(t)+y_p(t)-e_p(t))^Ty_c(t)-\nu_c (w_2(t)+y_p(t)-e_p(t))^T(w_2(t)+y_p(t)-e_p(t))-\rho_c y^T_c(t)y_c(t)
	\\&= w_1^T(t)y_p(t)+w_2^T(t)y_c(t)-\nu_pw_1^T(t)w_1(t)-\nu_cw_2^T(t)w_2(t)-(\rho_p+\nu_c)y^T_p(t)y_p(t)-(\rho_c+\nu_p)y^T_c(t)y_c(t)\\&+2\nu_pw^T_1(t)y_c(t)-2\nu_cw^T_2(t)y_p(t)+2\nu_cy^T_p(t)e_p(t)+2\nu_cw^T_2(t)e_p(t)-y^T_c(t)e_p(t)-\nu_ce_p^T(t)e_p(t)
	\\&+2\nu_py^T_c(t)e_c(t) -2\nu_pw^T_1(t)e_c(t)+y^T_p(t)e_c(t)-\nu_pe_c^T(t)e_c(t)
	\end{align*} \\
	Given that $2\nu_cw^T_2(t)e_p(t) \leq |\nu_c|w_2^T(t)w_2(t)+|\nu_c|e_p^T(t)e_p(t)$ and $-2\nu_pw^T_1(t)e_c(t) \leq |\nu_p|w_1^T(t)w_1(t)+|\nu_p|e_c^T(t)e_c(t)$ we have:\\
	
	$\dot{V}(t) \leq 2\begin{bmatrix}
	w^T_1(t) & w^T_2(t)      
	\end{bmatrix}\begin{bmatrix}
	\frac{1}{2} & \nu_p\\
	-\nu_c & \frac{1}{2}      
	\end{bmatrix}\begin{bmatrix}
	y_p(t) \\ y_c(t)      
	\end{bmatrix}+\begin{bmatrix}
	w^T_1(t) & w^T_2(t)      
	\end{bmatrix}\begin{bmatrix}
	|\nu_p|-\nu_p & 0\\
	0 & |\nu_c|-\nu_c      
	\end{bmatrix}\begin{bmatrix}
	w_1(t) \\ w_2(t)      
	\end{bmatrix}\\
	+\begin{bmatrix}
	y^T_p(t) & y^T_c(t)      
	\end{bmatrix}\begin{bmatrix}
	-(\rho_p+\nu_c) & 0\\
	0 & -(\rho_c+\nu_p)     
	\end{bmatrix}\begin{bmatrix}
	y_p(t) \\ y_c(t)      
	\end{bmatrix}+2\nu_cy^T_p(t)e_p(t)+2\nu_cw^T_2(t)e_p(t)-y^T_c(t)e_p(t)\\-\nu_ce_p^T(t)e_p(t)+2\nu_py^T_c(t)e_c(t) -2\nu_pw^T_1(t)e_c(t)+y^T_p(t)e_c(t)-\nu_pe_c^T(t)e_c(t)$\\
	
	Additionally we have: $-y^T_c(t)e_p(t)=-( e_p(t) + \frac{1}{2} y_c(t))^2+ e_p^T(t)e_p(t) + 
	\frac{1}{4} y_c^T(t)y_c(t)$ , and $e_p^T(t)e_p(t) \leq \delta_p y_p^T(t)y_p(t)$,$+y^T_p(t)e_c(t)=-( e_c(t) - \frac{1}{2} y_p(t))^2+ e_c^T(t)e_c(t) + 
	\frac{1}{4} y_p^T(t)y_p(t)$ , and $e_c^T(t)e_c(t) \leq \delta_c y_c^T(t)y_c(t)$ so we have:

	$\dot{V}(t) \leq 2\begin{bmatrix}
	w^T_1(t) & w^T_2(t)      
	\end{bmatrix}\begin{bmatrix}
	\frac{1}{2} & \nu_p\\
	-\nu_c & \frac{1}{2}      
	\end{bmatrix}\begin{bmatrix}
	y_p(t) \\ y_c(t)      
	\end{bmatrix}+\begin{bmatrix}
	w^T_1(t) & w^T_2(t)      
	\end{bmatrix}\begin{bmatrix}
	|\nu_p|-\nu_p & 0\\
	0 & |\nu_c|-\nu_c      
	\end{bmatrix}\begin{bmatrix}
	w_1(t) \\ w_2(t)      
	\end{bmatrix}\\
	+\begin{bmatrix}
	y^T_p(t) & y^T_c(t)      
	\end{bmatrix}\begin{bmatrix}
	-(\rho_p+\nu_c-\delta_p-\frac{1}{4}) & 0\\
	0 & -(\rho_c+\nu_p-\delta_c-\frac{1}{4})     
	\end{bmatrix}\begin{bmatrix}
	y_p(t) \\ y_c(t)      
	\end{bmatrix}+2\nu_cy^T_p(t)e_p(t)+|\nu_c|e_p^T(t)e_p(t)\\-\nu_ce_p^T(t)e_p(t)+2\nu_py^T_c(t)e_c(t)+|\nu_p|e_c^T(t)e_c(t)-\nu_pe_c^T(t)e_c(t)~~~~~~~~~~~~~~~~~~~~~~~~~~~~~~~~~~~~~(19)\\$
	
	Given the fact that both $e_p^T(t)e_p(t) \leq \delta_p y_p^T(t)y_p(t)$, and $e_c^T(t)e_c(t) \leq \delta_c y_c^T(t)y_c(t)$ always hold for any $t\in [t_{p_k},t_{p_{k+1}}) \cup [t_{c_k},t_{c_{k+1}})$, if the triggering conditions are chosen based on the relations $||e_p(t)||_2^2>\delta_p ||y_p(t)||_2^2$ with $\delta_p  \in(0,1]$, and $||e_c(t)||_2^2> \delta_c ||y_c(t)||_2^2$ with $\delta_c  \in(0,1]$, and by taking the same stepped mentioned in previous proofs, one can easily see that  depending on the signs of $\nu_p$, and $\nu_c$ the entire system is \emph{QSR-dissipative}: 
	\begin{align*}
	&\dot{V}(t) \leq w^T(t)Rw(t)+2w^T(t)Sy(t)+y^T(t)Qy(t),
	\end{align*}
	where $w(t)=\begin{bmatrix}
	w_1(t) \\ w_2(t)   
	\end{bmatrix}$,
	$y(t)=\begin{bmatrix}
	y_p(t) \\ y_c(t)   
	\end{bmatrix}$ and:\\
	
	$S=\begin{bmatrix}
	\frac{1}{2}I & \nu_pI\\
	-\nu_cI & \frac{1}{2}I      
	\end{bmatrix}$, $R=\begin{bmatrix}
	-(\nu_p-|\nu_p|)I & 0I\\
	0I & -(\nu_c-|\nu_c|)I      
	\end{bmatrix}$, and $Q=\begin{bmatrix}
	-(\beta(\nu_c)-\frac{1}{4})I & 0I\\
	0I & -(\beta(\nu_p)-\frac{1}{4})I\end{bmatrix}$
	and:\\
	
	\[\beta(\nu_p) =
	\begin{cases}
	\rho_c-\nu_p-\delta_c(1+\nu_p)      & \quad \text{if } \nu_p > 0\\
	\rho_c-\delta_c       & \quad \text{if } \nu_p = 0\\
	\rho_c+2\nu_p-\delta_c(1-3\nu_p)  & \quad \text{if } \nu_p < 0\\
	\end{cases}
	\]
	
	\[\beta(\nu_c) =
	\begin{cases}
	\rho_p-\nu_c-\delta_p(1+\nu_c)      & \quad \text{if } \nu_c > 0\\
	\rho_p-\delta_p       & \quad \text{if } \nu_c = 0\\
	\rho_p+2\nu_c-\delta_p(1-3\nu_c)  & \quad \text{if } \nu_c < 0\\
	\end{cases}
	\]
\end{proof}	
\subsubsection{Simulation Examples for \emph{QSR-dissipativity} for NCS (Event-triggering condition on the Plant's Output and Controller's Output - Figure \ref{fig:ETNCSPSCS})}
\begin{example}
	The plant in figure \ref{fig:NCSYpEpYcEcEX1} is defined by the following model:	
	
	\begin{figure}[!t]
		\centering
		\includegraphics[scale = .65]{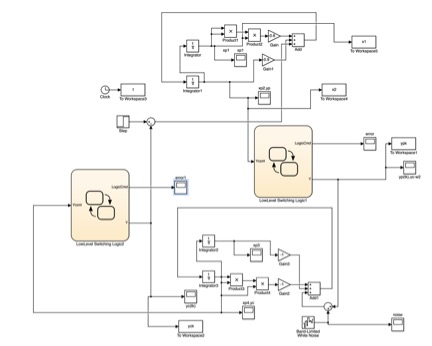}
		\caption{Simulink model for example 7.}
		\label{fig:NCSYpEpYcEcEX1}
	\end{figure}
	
	\begin{align*} 
	&\dot{x}_{p1}(t)=x_{p2}(t)\\
	&\dot{x}_{p2}(t)=-0.6x^3_{p1}(t)-0.9x_{p2}(t)+u_p(t)\\
	&y_p(t)=x_{p2}(t),
	\end{align*}
	
	where $\rho_p=0.9$ and $\nu_p=0$. And the model for the controller is the following:
	
	\begin{align*} 
	&\dot{x}_{p1}(t)=x_{p2}(t)\\
	&\dot{x}_{p2}(t)=-x_{p1}(t)-x^3_{p2}(t)+u_p(t)\\
	&y_p(t)=x_{p2}(t),
	\end{align*}
	
	where $\rho_c=1$ and $\nu_c=0$. The interconnection satisfies $\rho_c-\delta_c-\frac{1}{4}>0$, and $\rho_p-\delta_p - \frac{1}{4}>0$ by picking $\delta_p=0.6$, and $\delta_c=0.7$. $w1$ is a step signal, and $w2$ is white noise with power $0.02$, the simulation results are given in: figures  \ref{fig:NCSYpEpYcEcEX1XP1}, \ref{fig:NCSYpEpYcEcEX1XP2}, \ref{fig:NCSYpEpYcEcEX1Yc},\ref{fig:NCSYpEpYcEcEX1Yp}.
	\begin{figure}[!t]
		\centering
		\includegraphics[scale = .5]{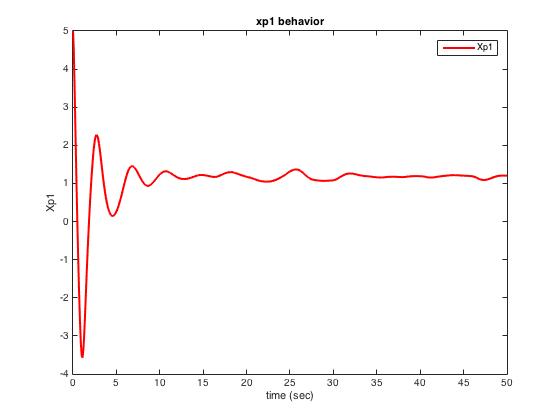}
		\caption{Simulation results for example 7.}
		\label{fig:NCSYpEpYcEcEX1XP1}
	\end{figure}
	\begin{figure}[!t]
		\centering
		\includegraphics[scale = .5]{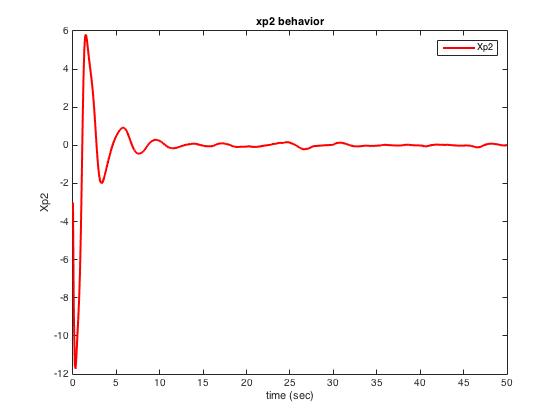}
		\caption{Simulation results for example 7.}
		\label{fig:NCSYpEpYcEcEX1XP2}
	\end{figure}
	\begin{figure}[!t]
		\centering
		\includegraphics[scale = .5]{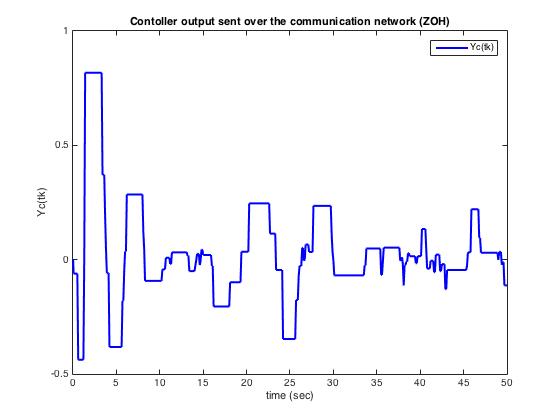}
		\caption{Simulation results for example 7.}
		\label{fig:NCSYpEpYcEcEX1Yc}
	\end{figure}
	\begin{figure}[!t]
		\centering
		\includegraphics[scale = .5]{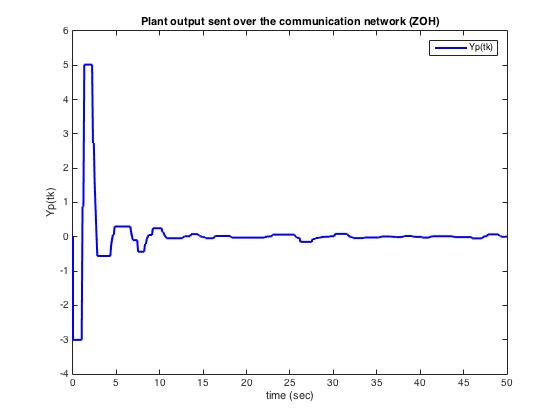}
		\caption{Simulation results for example 7.}
		\label{fig:NCSYpEpYcEcEX1Yp}
	\end{figure}
\end{example}
\begin{example}
	The plant is defined by the following transfer function:
	
	\begin{align*} 
	&G_p(s)=\frac{0.05s^2+2.1s+1.1}{s^2+2s+2}
	\end{align*}
	where $\rho_p=.80$ and $\nu_p=0.02$. And the model for the controller is the following:
	\begin{align*} 
	&\dot{x}_{c1}(t)=-x_{c1}(t)+u_c(t)\\
	&y_c(t)=-0.5x_{c1}(t)+1\\,
	\end{align*}
	where $\rho_c=1$ and $\nu_c=0.5$. The interconnection satisfies $\rho_c-\nu_p-\delta_c(1+\nu_p)-\frac{1}{4}>0$ and $\rho_p-\nu_c-\delta_p(1+\nu_c)-\frac{1}{4}>0$ by picking $\delta_p=0.02$ and $\delta_c=0.7$. $w1$ is a step signal, and $w2$ is white noise with power $0.02$, the simulation results are given in: figures  \ref{fig:NCSYpEpYcEcEX2Y}, \ref{fig:NCSYpEpYcEcEX2Yp},\ref{fig:NCSYpEpYcEcEX2Yc}.
	\begin{figure}[!t]
		\centering
		\includegraphics[scale = .5]{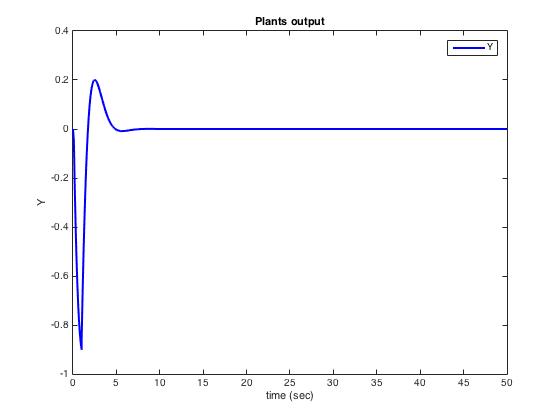}
		\caption{Simulation results for example 8.}
		\label{fig:NCSYpEpYcEcEX2Y}
	\end{figure}
	\begin{figure}[!t]
		\centering
		\includegraphics[scale = .5]{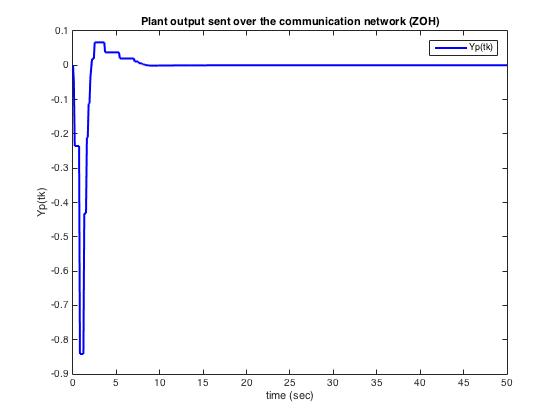}
		\caption{Simulation results for example 8.}
		\label{fig:NCSYpEpYcEcEX2Yp}
	\end{figure}
	\begin{figure}[!t]
		\centering
		\includegraphics[scale = .5]{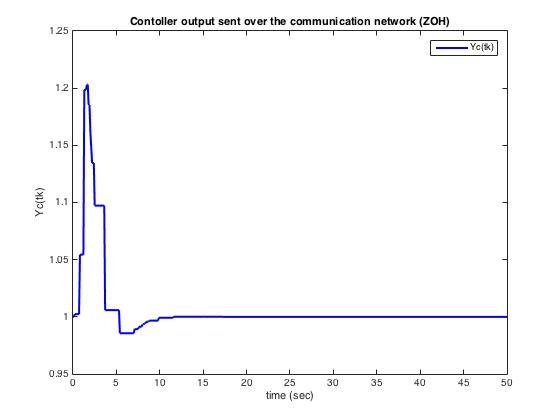}
		\caption{Simulation results for example 8.}
		\label{fig:NCSYpEpYcEcEX2Yc}
	\end{figure}
\end{example}
\subsubsection{Calculating Passivity indices for figure \ref{fig:ETNCSPSCS}  (Event-triggering condition on the Plant's Output and Controller's Output)}
\begin{theorem}

	The networked control system given in figure \ref{fig:ETNCSPSCS} where $G_p$ and $G_c$ have passivity indices $\nu_p$, $\rho_p$, $\nu_c$, and $\rho_c$ and the triggering instances $t_{p_k}$ and $t_{c_k}$ are explicitly determined by the conditions $||e_p(t)||_2^2>\delta_p ||y_p(t)||_2^2$ with $\delta_p  \in(0,1]$, and $||e_c(t)||_2^2>\delta_c ||y_c(t)||_2^2$ with $\delta_c  \in(0,1]$, is passive from the inputs $\begin{bmatrix}
	w_1(t) \\ w_2(t)      
	\end{bmatrix}$ to the outputs
	$\begin{bmatrix}
	y_p(t) \\ y_c(t)      
	\end{bmatrix}$ with input passivity index $\epsilon_0$ and output passivity index $\delta_0$:\\
	
	$\dot{V}(t) \leq 
	\begin{bmatrix}
	w^T_1(t) & w^T_2(t)      
	\end{bmatrix}
	\begin{bmatrix}
	y_p(t) \\ y_c(t)      
	\end{bmatrix}
	- \epsilon_0 
	\begin{bmatrix}
	w^T_1(t) & w^T_2(t)      
	\end{bmatrix}
	\begin{bmatrix}
	w_1(t) \\ w_2(t)      
	\end{bmatrix}
	-\delta_0 
	\begin{bmatrix}
	y_p^T(t) & y_c^T(t)      
	\end{bmatrix}
	\begin{bmatrix}
	y_p(t) \\ y_c(t)      
	\end{bmatrix}$,
	
	where:
	\begin{align*}
	&\epsilon_0 < \min(\nu_c-|\nu_c|,\nu_p-|\nu_p|)\\
	&\delta_0 < \min(\beta(\nu_c)-\frac{1}{4}-\frac{\nu_c^2}{\nu_c-|\nu_c|-\epsilon_0},\beta(\nu_p)-\frac{1}{4}-\frac{\nu_p^2}{\nu_p-|\nu_p|-\epsilon_0}).
	\end{align*}
	
\end{theorem}
\begin{proof}
	We want to calculate the passivity indices for the setup in figure \ref{fig:ETNCSPSCS} from the inputs $[w_1(t)~ w_2(t)]^T$ to outputs $[y_p(t)~ y_c(t)]^T$. We know that the setup is \emph{QSR-disspative} such that:
	
	\begin{align*}
	&\dot{V}(t) \leq w^T(t)Rw(t)+2w^T(t)Sy(t)+y^T(t)Qy(t)
	\end{align*}
	
	where	$S=\begin{bmatrix}
	\frac{1}{2}I & \nu_pI\\
	-\nu_cI & \frac{1}{2}I      
	\end{bmatrix}$, $R=\begin{bmatrix}
	-(\nu_p-|\nu_p|)I & 0I\\
	0I & -(\nu_c-|\nu_c|)I      
	\end{bmatrix}$, and $Q=\begin{bmatrix}
	-(\beta(\nu_c)-\frac{1}{4})I & 0I\\
	0I & -(\beta(\nu_p)-\frac{1}{4})I\end{bmatrix}$
	and:\\
	
	\[\beta(\nu_p) =
	\begin{cases}
	\rho_c-\nu_p-\delta_c(1+\nu_p)      & \quad \text{if } \nu_p > 0\\
	\rho_c-\delta_c       & \quad \text{if } \nu_p = 0\\
	\rho_c+2\nu_p-\delta_c(1-3\nu_p)  & \quad \text{if } \nu_p < 0\\
	\end{cases}
	\]
	
	\[\beta(\nu_c) =
	\begin{cases}
	\rho_p-\nu_c-\delta_p(1+\nu_c)      & \quad \text{if } \nu_c > 0\\
	\rho_p-\delta_p       & \quad \text{if } \nu_c = 0\\
	\rho_p+2\nu_c-\delta_p(1-3\nu_c)  & \quad \text{if } \nu_c < 0\\
	\end{cases}
	\]\\
	
	we need to show that \\

	$\dot{V}(t) \leq w^T(t)Rw(t)+2w^T(t)Sy(t)+y^T(t)Qy(t)\\~~~~~~~~~~~~ \leq 
	\begin{bmatrix}
	w^T_1(t) & w^T_2(t)      
	\end{bmatrix}
	\begin{bmatrix}
	y_p(t) \\ y_c(t)      
	\end{bmatrix}
	- \epsilon_0 
	\begin{bmatrix}
	w^T_1(t) & w^T_2(t)      
	\end{bmatrix}
	\begin{bmatrix}
	w_1(t) \\ w_2(t)      
	\end{bmatrix}
	-\delta_0 
	\begin{bmatrix}
	y_p^T(t) & y_c^T(t)      
	\end{bmatrix}
	\begin{bmatrix}
	y_p(t) \\ y_c(t)      
	\end{bmatrix}$~~~(20)

	and calculate the passivity indices $\epsilon_0$, and $\delta_0$.\\
	
	Simplifying (20) and moving the terms to one side we have: \\
	
	$\dot{V}(t)\leq  -(\nu_c-|\nu_c|-\epsilon_0)w_2^T(t)w_2(t)-(\nu_p-|\nu_p|-\epsilon_0) w_1^T(t)w_1(t)-(\beta(\nu_p)-\frac{1}{4}-\delta_0)y^T_c(t)y_c(t)\\~~~~~~~~~~~~-(\beta(\nu_c)-\frac{1}{4}-\delta_0)y^T_p(t)y_p(t)+2\nu_pw^T_1(t)y_c(t)-2\nu_cw^T_2(t)y_p(t)\\~~~~~~~~~~~\leq 0. $\\
	
	So we need to show:
	
	$\begin{bmatrix}
	w^T_1(t) & y_c^T(t)      
	\end{bmatrix} M
	\begin{bmatrix}
	w_1(t) \\ y_c(t)      
	\end{bmatrix} + \begin{bmatrix}
	w^T_2(t) & y^T_p(t)      \end{bmatrix} N \begin{bmatrix}
	w_2(t) \\ y_p(t)      \end{bmatrix} \leq 0$\\

	where  $M=\begin{bmatrix}
	-(\nu_p-|\nu_p|-\epsilon_0) & \nu_p\\
	\nu_p & -(\beta(\nu_p)-\frac{1}{4}-\delta_0)
	\end{bmatrix}$ and  $N=\begin{bmatrix}
	-(\nu_c-|\nu_c|-\epsilon_0) & -\nu_c\\
	-\nu_c&-(\beta(\nu_c)-\frac{1}{4}-\delta_0) \end{bmatrix}$.\\ 
	
	For matrices M and N to be negative semi-definite, they need to meet the following conditions:
	
	\begin{align*}
	\epsilon_0 < \nu_c-|\nu_c|\\
	\epsilon_0 <  \nu_p-|\nu_p|\\
	\delta_0 < \beta(\nu_p)-\frac{1}{4}\\
	\delta_0 < \beta(\nu_c)-\frac{1}{4}\\
	\nu_p^2\leq (\nu_c-|\nu_c|-\epsilon_0)(\beta(\nu_c)-\frac{1}{4}-\delta_0)\\
	\nu_c^2\leq (\nu_p-|\nu_p|-\epsilon_0)(\beta(\nu_p)-\frac{1}{4}-\delta_0)
	\end{align*}
	
	which also prove the theorem.
\end{proof}
\subsubsection{Passivity for Figure \ref{fig:ETNCSPSCS} from $w_1 \to y_p$ (Event-triggering conditions on the Plant's Output and Controller's Output)}
\begin{theorem}
	
	The networked control system given in figure \ref{fig:ETNCSPSCS} where $G_p$ and $G_c$ have passivity indices $\nu_p$, $\rho_p$, $\nu_c$, and $\rho_c$ and triggering instances $t_{p_k}$ and $t_{c_k}$ are explicitly determined by the conditions $||e_p(t)||_2^2>\delta_p ||y_p(t)||_2^2$ with $\delta_p  \in(0,1]$, and $||e_c(t)||_2^2>\delta_c ||y_c(t)||_2^2$ with $\delta_c  \in(0,1]$, and $w_2=0$, is passive from $w_1 \to y_p$ meaning:\\
	\begin{align*}
	\dot{V}(t) \leq w_1^T(t)y_p(t)
	\end{align*}
	if:
	\begin{align*}
	\beta(\nu_p) \geq \frac{1}{4}\\
	\nu_p = 0\\
	\beta(\nu_c) \geq \frac{1}{4}\\
	\end{align*}	  
\end{theorem}	

\begin{proof}
	We know that the setup is \emph{QSR-disspative} such that:
	
	\begin{align*}
	&\dot{V}(t) \leq w^T(t)Rw(t)+2w^T(t)Sy(t)+y^T(t)Qy(t)
	\end{align*}
	
	where $w(t)=\begin{bmatrix}
	w_1(t) \\ w_2(t)   
	\end{bmatrix}$,
	$y(t)=\begin{bmatrix}
	y_p(t) \\ y_c(t)   
	\end{bmatrix}$ and:\\
	
	$S=\begin{bmatrix}
	\frac{1}{2}I & \nu_pI\\
	-\nu_cI & \frac{1}{2}I      
	\end{bmatrix}$, $R=\begin{bmatrix}
	-(\nu_p-|\nu_p|)I & 0I\\
	0I & -(\nu_c-|\nu_c|)I      
	\end{bmatrix}$, and $Q=\begin{bmatrix}
	-(\beta(\nu_c)-\frac{1}{4})I & 0I\\
	0I & -(\beta(\nu_p)-\frac{1}{4})I\end{bmatrix}$
	and:\\
	
	\[\beta(\nu_p) =
	\begin{cases}
	\rho_c-\nu_p-\delta_c(1+\nu_p)      & \quad \text{if } \nu_p > 0\\
	\rho_c-\delta_c       & \quad \text{if } \nu_p = 0\\
	\rho_c+2\nu_p-\delta_c(1-3\nu_p)  & \quad \text{if } \nu_p < 0\\
	\end{cases}
	\]
	
	\[\beta(\nu_c) =
	\begin{cases}
	\rho_p-\nu_c-\delta_p(1+\nu_c)      & \quad \text{if } \nu_c > 0\\
	\rho_p-\delta_p       & \quad \text{if } \nu_c = 0\\
	\rho_p+2\nu_c-\delta_p(1-3\nu_c)  & \quad \text{if } \nu_c < 0\\
	\end{cases}
	\]
	Given that $w_2=0$, and using the relation $2\nu_pw_1^T(t)y_c(t) \leq |\nu_p|y_c^T(t)y_c(t)+|\nu_p|w_1^T(t)w_1(t)$:  \\
	
	$\dot{V}(t)\leq -(\nu_p-|\nu_p|) w_1^T(t)w_1(t)-(\beta(\nu_p)-\frac{1}{4})y^T_c(t)y_c(t)-(\beta(\nu_c)-\frac{1}{4})y^T_p(t)y_p(t)+2\nu_pw^T_1(t)y_c(t)\\~~~~~~~~~~~~\leq-(\nu_p-2|\nu_p|) w_1^T(t)w_1(t)-(\beta(\nu_p)-\frac{1}{4}-|\nu_p|)y^T_c(t)y_c(t)-(\beta(\nu_c)-\frac{1}{4})y^T_p(t)y_p(t)\\~~~~~~~~~~~~ \leq 0 $

	it is easy to see that the system is passive if:
	\begin{align*}
	\beta(\nu_p) \geq \frac{1}{4}\\
	\nu_p = 0\\
	\beta(\nu_c) \geq \frac{1}{4}\\
	\end{align*}	 
\end{proof}
\subsubsection{Passivity and Passivity indices for Figure \ref{fig:ETNCSPSCS} from $w_1 \to y_p$ (Event-triggering conditions on the Plant's Output and Controller's Output)}
\begin{theorem}
	
	The networked control system given in figure \ref{fig:ETNCSPSCS} where $G_p$ and $G_c$ have passivity indices $\nu_p$, $\rho_p$, $\nu_c$, and $\rho_c$ and triggering instances $t_{p_k}$ and $t_{c_k}$ are explicitly determined by the conditions $||e_p(t)||_2^2>\delta_p ||y_p(t)||_2^2$ with $\delta_p  \in(0,1]$, and $||e_c(t)||_2^2>\delta_c ||y_c(t)||_2^2$ with $\delta_c  \in(0,1]$, and $w_2=0$, is passive from $w_1 \to y_p$ if $\nu_p=0$ with input passivity index $\epsilon_0$ and output passivity index $\delta_0$ meaning:\\
	\begin{align*}
	\dot{V}(t) \leq w_1^T(t)y_p(t)- \epsilon_0 w^T_1(t)w_1(t) -\delta_0 y_p^T(t)y_p(t) 
	\end{align*}
	
	where
	\begin{align*}
	\epsilon_0 =0 \\
	0\leq \delta_0 \leq \beta(\nu_c) - \frac{1}{4}\\\\
	\end{align*}	 
\end{theorem}

\begin{proof}
	We know that the setup is \emph{QSR-disspative} such that:
	
	\begin{align*}
	&\dot{V}(t) \leq w^T(t)Rw(t)+2w^T(t)Sy(t)+y^T(t)Qy(t)
	\end{align*}
	
	where $w(t)=\begin{bmatrix}
	w_1(t) \\ w_2(t)   
	\end{bmatrix}$,
	$y(t)=\begin{bmatrix}
	y_p(t) \\ y_c(t)   
	\end{bmatrix}$ and:\\
	
	$S=\begin{bmatrix}
	\frac{1}{2}I & \nu_pI\\
	-\nu_cI & \frac{1}{2}I      
	\end{bmatrix}$, $R=\begin{bmatrix}
	-(\nu_p-|\nu_p|)I & 0I\\
	0I & -(\nu_c-|\nu_c|)I      
	\end{bmatrix}$, and $Q=\begin{bmatrix}
	-(\beta(\nu_c)-\frac{1}{4})I & 0I\\
	0I & -(\beta(\nu_p)-\frac{1}{4})I\end{bmatrix}$
	and:\\
	
	\[\beta(\nu_p) =
	\begin{cases}
	\rho_c-\nu_p-\delta_c(1+\nu_p)      & \quad \text{if } \nu_p > 0\\
	\rho_c-\delta_c       & \quad \text{if } \nu_p = 0\\
	\rho_c+2\nu_p-\delta_c(1-3\nu_p)  & \quad \text{if } \nu_p < 0\\
	\end{cases}
	\]
	
	\[\beta(\nu_c) =
	\begin{cases}
	\rho_p-\nu_c-\delta_p(1+\nu_c)      & \quad \text{if } \nu_c > 0\\
	\rho_p-\delta_p       & \quad \text{if } \nu_c = 0\\
	\rho_p+2\nu_c-\delta_p(1-3\nu_c)  & \quad \text{if } \nu_c < 0\\
	\end{cases}
	\]
	Given that $w_2=0$, and using the relation $2\nu_pw_1^T(t)y_c(t) \leq |\nu_p|y_c^T(t)y_c(t)+|\nu_p|w_1^T(t)w_1(t)$: \\

	$\dot{V}(t)\leq -(\nu_p-|\nu_p|-\epsilon_0) w_1^T(t)w_1(t)-(\beta(\nu_p)-\frac{1}{4})y^T_c(t)y_c(t)-(\beta(\nu_c)-\frac{1}{4}-\delta_0)y^T_p(t)y_p(t)\\~~~~~~~~~~~+2\nu_pw^T_1(t)y_c(t)\\~~~~~~~~~~~\leq-(\nu_p-2|\nu_p|-\epsilon_0) w_1^T(t)w_1(t)-(\beta(\nu_p)-\frac{1}{4}-|\nu_p|)y^T_c(t)y_c(t)-(\beta(\nu_c)-\frac{1}{4}-\delta_0)y^T_p(t)y_p(t)\\~~~~~~~~~~~ \leq 0 $
	
	which is met if:
	\begin{align*}
	\beta(\nu_p) \geq \frac{1}{4}\\
	\epsilon_0=0\\
	\nu_p =0\\
	\beta(\nu_c) - \frac{1}{4}\geq \delta_0
	\end{align*}	
	
\end{proof}
\begin{example}
	For the feedback interconnection given in example 7, we know that $\rho_p=0.9$, $\nu_p=0$ and $\rho_c=1$ and $\nu_c=0.3$. By choosing $\delta_p = 0.55$, and $\delta_c=0.55$ all conditions given in theorem 12 hold, namely $\beta(\nu_p) \geq \frac{1}{4}$, $\epsilon_0=0$ and $\nu_p =0$. Accordingly $\delta_0 \leq 0.20 $ so we can calculate $\epsilon_0=0$, and $\delta_0 = 0.2$. Figure \ref{fig:PassivityYpEpYcEc} shows that the system is passive with output passivity index $\delta_0=0.20$ and input passivity index $\epsilon_0=0$ and the relation $\int_{0}^{T}(w_1^T(t)y_p(t)-0.2y_p^T(t)y_p(t)) dt$ holds.
	
	\begin{figure}[!t]
		\centering
		\includegraphics[scale = .4]{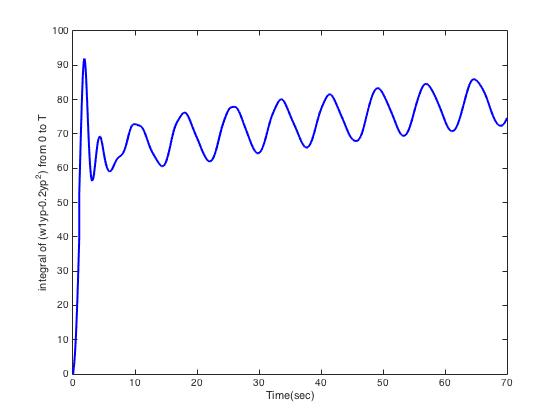}
		\caption{Simulation result for example 9 where $w_1=sin(\frac{5\pi}{2}t)+3$ .}
		\label{fig:PassivityYpEpYcEc}
	\end{figure}
\end{example}

\section{Conclusion}
\label{sec:con}
The work in this report was motivated by the belief that \emph{QSR-dissipativity} and passivity based control can address many challenges rising from the increasing demand for more scalable systematic control processes for large-scale cyber-physical systems. CPS are usually run remotely through network channels and require decentralized control methodologies rendering networked control schemes as great candidates for their design. Therefore, our work sought to bring together the control of CPS and NCS by utilizing passivity, and \emph{QSR-dissipativity} as a unifying control force. We believe that innate properties of \emph{QSR-dissipativity} and passivity such as high compositionality and their ability to bring stability to systems under simple conditions can solve challenges such as expandability, stability, robustness and reliability in networked control cyber-physical systems. Our design intended to solve one of the main issues in NCS, namely channel utilization by reducing the communication rate amongs sub-systems.  In our work, we showed \emph{QSR-dissipativity}, passivity and $L_2$-stability requirements for different event-triggered NCS platforms based on simple triggering conditions. Our work and design methodology considerably decrease the system's reliance on network channels and information exchange with controllers and other sub-units in NCS. We intend to expand our results in the future by working on passivation, data loss, and delays in NCS. Additionally, we would like to analyze networked control structures and their properties through approximating each sub-units' output and its respective event-detector.
\clearpage
\bibliographystyle{IEEEtran}
\bibliography{bibfile}	 
\end{document}